\documentclass[12pt]{amsart}

\usepackage[all,color]{xy}
\usepackage{pb-diagram}
\usepackage[mathscr]{eucal}
\usepackage{hyperref}
\hypersetup{
    colorlinks=true, 
    linktoc=all,     
    linkcolor=blue,  
}

\usepackage{xcolor}
\usepackage[normalem]{ulem}

\DeclareMathAlphabet{\mathpzc}{OT1}{pzc}{m}{it}

\usepackage{amsfonts}
\usepackage{amsmath}
\usepackage{amssymb}

\newcommand{\tr}{\textnormal{tr}}

\newcommand{\dbar}{\overline{\partial}}

\newcommand{\ddt}[1]{\frac{\partial #1}{\partial t}}

\newcommand{\ddbar}{\sqrt{-1}\partial\dbar}

\def\cI{{\mathcal I}}
\def\cK{{\mathcal K}}

\def\cM{{\mathcal M}}
\def\cJ{{\mathcal J}}

\def\cP{{\mathcal P}}

\def\cR{{\mathcal R}}
\def\cS{{\mathcal S}}
\def\cT{{\mathcal T}}
\def\cV{{\mathcal V}}
\def\cX{{\mathcal X}}

\def\cZ{{\mathcal Z}}

\def\C{{\mathbb C}}

\def\PSH{\textnormal{PSH}}
\def\Vol{\textnormal{Vol}}

\newtheorem{theorem}{Theorem}[section]
\newtheorem{proposition}{Proposition}[section]
\newtheorem{lemma}{Lemma}[section]
\newtheorem{definition}{Definition}[section]
\newtheorem{corollary}{Corollary}[section]

\newtheorem{conjecture}{Conjecture}[section]

\numberwithin{equation}{section}

\begin{document}

\centerline{{\bf  NAKAI-MOISHEZON CRITERIONS FOR COMPLEX HESSIAN EQUATIONS } \footnote{Research supported in part by National Science Foundation grant  DMS-1711439 }}

\bigskip

\address{Department of Mathematics, Rutgers University, Piscataway, NJ 08854}

\medskip

\centerline{ \small  JIAN SONG}

 \bigskip
 \bigskip

{\noindent \small A{\scriptsize BSTRACT}. \footnotesize  The $J$-equation proposed by Donaldson is a complex Hessian quotient equation on K\"ahler manifolds. The solvability of the $J$-equation is proved by Song-Weinkove to be equivalent to the existence of a subsolution. It is also conjectured by Lejmi-Szekelyhidi to be equivalent to a stability condition in terms of holomorphic intersection numbers as an analogue of the Nakai-Moishezon criterion in algebraic geometry. The conjecture is recently proved by Chen under a stronger uniform stability condition. In this paper, we establish a Nakai-Moishezon type criterion for pairs of K\"ahler classes on analytic K\"ahler varieties. As a consequence,  we prove Lejmi-Szekelyhidi's original conjecture for the $J$-equation. We also apply such a criterion to obtain a family of constant scalar curvature K\"ahler metrics on smooth minimal models.}


\bigskip
\bigskip

\section{Introduction}


The $J$-equation is the critical equation of the $J$-flow introduced by Donaldson \cite{Do} in the framework of moment maps. Let $X$ be an $n$-dimensional compact K\"ahler manifold with two K\"ahler classes $\alpha$ and $\beta$ satisfying the normalization condition
\begin{equation}\label{normalcon}
\alpha^n = \alpha^{n-1}\cdot \beta, 
\end{equation}
or equivalently
\begin{equation}\label{intersec}
\int_X \alpha^n = \int_X \alpha^{n-1} \wedge \beta,
\end{equation}
where the integral on the right hand side of (\ref{intersec}) is calculated by choosing any K\"ahler forms in $\alpha$ and $\beta$.
For a given K\"ahler form $\chi \in \beta$, the $J$-equation is defined by
\begin{equation}\label{jeqn}
\omega^n = \omega^{n-1} \wedge \chi, 
\end{equation}
where $\omega\in \alpha$ is the desired K\"ahler form. The $J$-equation (\ref{jeqn}) can also be expressed as the complex Hessian quotient equation
$$tr_{\omega}(\chi) =  \frac{n \omega^{n-1}\wedge \chi}{\omega^n} = n .$$
At any given point $p\in X$, we can diagonalize both $\chi$ and $\omega$ such that $\chi$ is an identity matrix of size $n\times n$ and $\omega$ has positive eigenvalues $\lambda_1, \lambda_2, ..., \lambda_n$. Then equation (\ref{jeqn}) becomes 
$$\frac{\sigma_{n-1}(\Lambda)}{\sigma_{n}(\Lambda)} = n, $$
where $\Lambda=\{\lambda_j\}_{j=1}^n$ and $\sigma_k$ is the $k$-th elementary symmetric polynomial of $\Lambda$ of degree $k$. 

Both the $J$-equation and the $J$-flow are extensively studied in \cite{Ch1, Ch2, W1, W2, SW1, FLM, SW2, LS, CS, Sz, FL, ChG}. The convergence of the $J$-flow is first proved by Weinkove in \cite{W1} to be equivalent to the class condition $2\alpha - \beta>0$  for K\"ahler surfaces. The general case is proved by Song-Weinkove in \cite{SW1} by the $J$-flow that (\ref{jeqn}) admits a smooth solution $\omega$ if and only if there exists a K\"ahler form $\Omega\in \alpha$ such that
\begin{equation}\label{jsw}
n\Omega^n - (n-1)\Omega^{n-1}\wedge \chi >0
\end{equation}
as a positive $(n-1, n-1)$-form on $X$. The analytic positive condition (\ref{jsw}) is later interpreted in \cite{Sz} as a subsolution of the $J$-equation.  This immediately implies that there is obstruction to solve the global $J$-equation. There are many well-known examples such as certain pairs of K\"ahler classes on Hirzebruch surfaces for which the $J$-equation does not admit smooth solutions  (c.f. \cite{FL}).

The pointwise condition (\ref{jsw}) is however difficult to verify, and one would like to replace it by a holomorphic/topological condition. Of course,    (\ref{jsw}) always holds whenever  $n\alpha - (n-1) \beta>0$ is a K\"ahler class \cite{W2}.  This leads to following conjecture proposed by Lejmi-Szekelyhidi \cite{LS}. 
\begin{conjecture} \label{conls} Under the normalization condition (\ref{normalcon}),  the $J$-equation (\ref{jeqn}) admits a unique smooth solution  $\omega$
if and only for any $m$-dimensional analytic subvariety $Z$ of $X$ with $1\leq m \leq n-1$, 
\begin{equation}\label{conls1} 
\left( n\alpha^m - m \alpha^{m-1}\cdot \beta \right)\cdot Z = \int_Z (n\alpha^m - m\alpha^{m-1}\wedge \beta) >0. 
\end{equation}

\end{conjecture}

 Conjecture \ref{conls} is proved for toric K\"ahler manifolds in \cite{CS}. Recently, a major progress is made toward this conjecture by Chen \cite{ChG}, proving Conjecture \ref{conls}  under a slightly stronger topological condition (uniform $J$-positivity, see Definition \ref{jamp} (3)). Chen's result is generalized by Datar-Pingali \cite{DaP} to a more general family of complex Hessian equations. Conjecture \ref{conls} is also proved in \cite{DaP} for projective K\"ahler manifolds and positive line bundles.  

The main goal of this paper is to prove Conjecture \ref{conls} and it suffices to construct a subsolution satisfying the analytic positive condition (\ref{jsw}). It is crucial to find the connection between (\ref{jsw}) and   (\ref{conls1}) as an analogue of the Nakai-Moishezon criterion for a K\"ahler class on a K\"ahler manifold established by Demailly-Paun \cite{DP}. This leads us to prove a Nakai-Moishezon type criterion for the $J$-equation. The following is the main result of the paper.

\begin{theorem} \label{main1} Let $X$ be an $n$-dimensional compact analytic variety embedded in  a K\"ahler manifold $\cM$. Let $\alpha$ and $\beta$ be two K\"ahler classes in an open neighborhood of $X$ in $\cM$  satisfying
$$ \left. \frac{  \alpha^{n-1}\cdot \beta  }{ \alpha^n } \right|_Y=  \frac{ \int_Y \alpha^{n-1}\wedge \beta }{ \int_Y \alpha^n  }  \leq 1$$
for any component $Y$ of $X$.
Then the following two conditions are equivalent.

\begin{enumerate}

\item For any $m$-dimensional analytic subvariety $Z$  of $X$, 
$$
\left( n  \alpha^m - m\alpha^{m-1}\cdot \beta \right) \cdot  Z = \int_Z \left( n  \alpha^m - m\alpha^{m-1}\wedge \beta \right) >0, 
$$
where $1\leq m <n$. 

\medskip

\item For any smooth K\"ahler form $\chi\in \beta$ in an open neighborhood $U$ of $X$, there exists a smooth K\"ahler form $\omega\in \alpha$ in some open neighborhood $V \subset U$ of $X$ such that in $V$, 
$$n \omega^{n-1} - (n-1) \omega^{n-2}\wedge \chi >0. $$

\end{enumerate}

\end{theorem}

In Theorem \ref{main1}, $\cM$ can be open or incomplete. In particular,  we can let $\cM = X$ if $X$ is a K\"ahler manifold, and the following corollary is immediate after suitable rescaling.

\begin{corollary} \label{main2} Let $X$ be an $n$-dimensional compact K\"ahler manifold and let $\alpha$ and $\beta$ be two K\"ahler classes on $X$. If for any $m$-dimensional analytic subvariety $Z$  of $X$ with $1\leq m \leq n-1$, 
\begin{equation}\label{slope1}
m\left. \frac{ \alpha^{m-1}\cdot \beta }{  \alpha^m  } \right|_Z<  n\frac{ \alpha^{n-1}\cdot \beta } {    \alpha^n} , 
\end{equation}
then  for any smooth K\"ahler form $\chi\in \beta$, there exists a smooth K\"ahler form $\omega\in \alpha$ such that at any point $p\in X$ and $m$-dimensional subspace $W$ of $T_pX$
\begin{equation} \label{slope2}
m\left. \frac{ \omega^{m-1}\wedge \chi}{ \omega^{m} } \right|_W <   n \frac{  \alpha^{n-1}\cdot \beta } {    \alpha^n  }.
\end{equation} 
\end{corollary}
 
We remark that 
$$
tr_{\omega|_W}(\chi|_W) = m\left. \frac{ \omega^{m-1}\wedge \chi}{ \omega^{m} } \right|_W $$ for $1\leq m \leq n-1$ and (\ref{slope2}) is equivalent to 
$$n \frac{  \alpha^{n-1}\cdot \beta } {\alpha^n  } \omega^{n-1} - (n-1) \omega^{n-2}\wedge \chi >0$$
as an $(n-1,n-1)$-form on $X$.

 The condition (\ref{slope1}) is analogous to the topological slope condition for Hermitian vector bundles over K\"ahler manifolds, while condition (\ref{slope2}) is the corresponding pointwise slope condition. Therefore, we define the following positive conditions as an analogue of the Nakai-Moishezon Criterion.

\begin{definition} \label{jamp} Let $X$ be an $n$-dimensional compact K\"ahler manifold with two K\"ahler classes $\alpha$ and $\beta$.

\begin{enumerate}

\item The pair $(\alpha, \beta)$ is said to be $J$-positive if for any $m$-dimensional analytic subvariety $Z$ of $X$, 

$$
m\left. \frac{ \alpha^{m-1}\cdot \beta }{  \alpha^m  } \right|_Z<  n\frac{ \alpha^{n-1}\cdot \beta } {    \alpha^n}. 
$$
\item The pair is  said to be $J$-nef if for any $m$-dimensional subvariety $Z$ of $X$, 

$$
m\left. \frac{ \alpha^{m-1}\cdot \beta }{  \alpha^m  } \right|_Z \leq  n\frac{ \alpha^{n-1}\cdot \beta } {    \alpha^n} . 
$$

\item The pair is  said to be uniformly $J$-positive if there exists $\varepsilon>0$ such that for any $m$-dimensional analytic subvariety $Z$ of $X$, 

$$
m\left. \frac{ \alpha^{m-1}\cdot \beta }{  \alpha^m  } \right|_Z \leq  (n-\varepsilon) \frac{ \alpha^{n-1}\cdot \beta } {    \alpha^n} . 
$$

\end{enumerate}

\end{definition}

Obviously, $(3) \Rightarrow (1) \Rightarrow (2)$ and Definition \ref{jamp} can be generalized to analytic K\"ahler varieties.  The uniform $J$-positive condition (3) is introduced and proved by Chen  \cite{ChG} to be equivalent to the solvability of the $J$-equation. Our next result settles the original conjecture of Lejmi-Szekleyhidi (Conjecture \ref{conls}).

\begin{corollary} \label{main3} Let $X$ be an $n$-dimensional  K\"ahler manifold with two K\"ahler classes $\alpha$ and $\beta$.  For any K\"ahler form $\chi\in \beta$, the J-equation
\begin{equation} \label{jeqn2}
\frac{\omega^{n-1}\wedge \chi }{\omega^n }=  \frac{\alpha^{n-1}\cdot \beta }{\alpha^n}
\end{equation}
admits a unique smooth solution  $\omega\in \alpha$ if and only if the pair $(\alpha, \beta)$ is $J$-positive as in Definition \ref{jamp}. 

\end{corollary}

The proof of Theorem \ref{main1} follows the road map in \cite{ChG} based on the techniques from \cite{DP}. The key idea of \cite{ChG} is to apply Blocki-Kolodziej's gluing trick by local regularization \cite{BK} instead of the global regularization  in \cite{DP}. The approach of \cite{ChG} relies on the induction argument for smooth subvarieties of $X$. Our approach removes the smoothness assumption and has to work on the degenerate $J$-equation instead. We expect that Theorem \ref{main1} can be generalized to a family of complex Hessian quotient equations by applying the work in \cite{DaP}.

We would also like to extend Conjecture \ref{conls} to the singular case. 

\begin{conjecture} Let $X$ be an $n$-dimensional  normal K\"ahler variety with two K\"ahler classes $\alpha$ and $\beta$. If $(\alpha, \beta)$ is $J$-positive, then  for any K\"ahler form $\chi\in \beta$, the J-equation
$$
\omega^{n-1}\wedge \chi  =  \left( \frac{\alpha^{n-1}\cdot \beta }{\alpha^n} \right) \omega^n
$$
admits a unique solution $\omega$ as a K\"ahler current  with bounded local potentials. 

\end{conjecture}

We can further extend the above conjecture to any $J$-nef pair $(\alpha, \beta)$ on $X$. We expect that the solution $\omega$ should be a  K\"ahler current with vanishing Lelong number similar to degenerate complex Monge-Ampere equations studied in \cite{BEGZ}.  A viscosity approach in \cite{DDT} might also be helpful to understand weak solutions of the degenerate global $J$-equation. We would like to remark that Theorem \ref{main1} and Definition \ref{jamp} can indeed be generalized for certain special degenerate case. For example, if $\beta$ is a semi-positive class (i.e., there exists a smooth closed semi-positive $(1,1)$-form in $\beta$) satisfying $\alpha^{n-1}\cdot \beta>0$, then  Theorem \ref{main1} should still hold.  

The more interesting  question is what canonical solutions one expects for the $J$-equation if the $J$-positive condition fails. The natural approach is to apply the $J$-flow, especially for projective manifolds with large symmetry, as suggested by the author and   investigated by Fang-Lai \cite{FL}. The interesting examples constructed in \cite{FL} suggest a natural analogy between the $J$-flow and the Yang-Mills flow in terms of formation of singularities.  Let $\omega \in \alpha$ and $\chi\in \beta$ be two K\"ahler forms on an $n$-dimensional compact K\"ahler manifold $X$. Then the $J$-flow is defined as below. 
\begin{equation}
\ddt{\varphi} = c- \frac{(\omega+\ddbar \varphi)^{n-1} \wedge \chi}{(\omega+ \ddbar \varphi)^n}, ~~\varphi(0) =0,
\end{equation}
where $c= \frac{\alpha^{n-1}\cdot \beta}{\alpha^n}$.  We conjecture that  the $J$-flow always converges smoothly outside an analytic subvariety $Z$ of $X$. The limiting solution $\omega_\infty$ on $X\setminus Z$ satisfies a $J$-equation 
$$\frac{\omega_\infty^{n-1}\wedge \chi}{ \omega_\infty^n }= c'>0,$$
where $c'$ is not necessarily equal to $c$ unless the pair $(\alpha, \beta)$ is $J$-nef. Furthermore, the limiting solution $\omega_\infty$ uniquely extends to a global K\"ahler current  on $X$. A suitable restriction of this K\"ahler current on each component of $Z$ would still satisfy a $J$-equation, possibly degenerate. Such a phenomena is possibly related to the slope stability through a  filtration of analytic subvarieties of $X$.

Finally, we extend the result of \cite{W2, SW1, JSS} for constant scalar curvature K\"ahler metrics on smooth minimal models.  

\begin{theorem} \label{maincsc} Let $X$ be an $n$-dimensional smooth minimal model, i.e., $X$ is a compact K\"ahler manifold with nef $K_X$. Let $\gamma$ be  a K\"ahler class and $\alpha(X, \gamma)$ be the $\alpha$-invariant (c.f. (\ref{alphainv})) for $\gamma$ on $X$. If there exists $\epsilon \in \left[0, \frac{n+1}{n} \alpha(X, \gamma)\right)$ such that   for any $m$-dimensional analytic subvariety $Z$ of $X$
\begin{equation}\label{slopecsc}
 m\left. \frac{ \gamma^{m-1}\cdot K_X }{  \gamma^m  } \right|_Z \leq  ( n + (n-m) \epsilon)\frac{ \gamma^{n-1}\cdot K_X } {    \gamma^n},  
 \end{equation}
 then there exists a unique cscK metric  (constant scalar curvature K\"ahler metric) in $\gamma$.

\end{theorem}

Theorem \ref{maincsc} immediately implies the result of \cite{JSS, SD} on the existence of cscK metrics near the canonical class on a smooth minimal model because any sufficiently small perturbation of $K_X$ by a K\"ahler class will satisfy the assumption in Theorem \ref{maincsc}.

\begin{corollary} \label{maincsc2} Let $X$ be an $n$-dimensional smooth minimal model. Then for any K\"ahler class $\gamma$, there exists a cscK metric in $K_X + \epsilon \gamma$ for any sufficiently small $\epsilon>0$.

\end{corollary}

 Let us give a brief outline of the paper. We first give a direct proof of Conjecture \ref{conls} for dimension $3$ in \S2 by combining the work of Chen \cite{ChG} and solutions to degenerate complex Monge-Ampere equations.    In \S3, we will start the proof of Theorem \ref{main1} with basic set-up.  The main ingredients of the proof for Theorem \ref{main1} are established in \S 4, \S5 and \S6. We complete the proof for the main results of the paper in \S7. Finally, we prove Theorem \ref{maincsc} in  \S8.


\section{Lejmi-Szekelyhidi's conjecture for  $\dim =3$ }

In this section, we will directly prove Corollary \ref{main3} for dimension $3$ without using Theorem \ref{main1}. The proof will be different from the general case. Our approach here  is to solve degenerate $J$-equations with prescribed singularities on K\"ahler surfaces.  Then the argument of Chen (\cite{ChG}) can be directly extend to the case of $J$-positivity  from uniformly $J$-positivity and Lejmi-Szekelyhidi's original conjecture immediately follows. We rephrase Corollary \ref{main3} in the case of $\dim X=3$ as below. 

\begin{theorem} \label{s2nthmain} Let $X$ be a $3$-dimensional K\"ahler manifold and let $\alpha$ and $\beta$ be two K\"ahler classes on $X$ with $\alpha^2 \cdot \beta = \alpha^3$. If the pair $(\alpha, \beta)$ is $J$-positive (i.e., for any proper subvariety $Z$ of $X$,
$$(3\alpha - \beta)\cdot Z>0$$ if $\dim Z=1$ and 
$$ \left( 3 \alpha^2 - 2 \alpha \cdot \beta\right)\cdot Z>0$$ if $\dim Z=2$),
 then for any K\"ahler form $\chi\in \beta$, there exists a K\"ahler form $\omega\in \alpha$ satisfying the $J$-equation
\begin{equation}\label{j3d}
\omega^2\wedge \chi = \omega^3
\end{equation}
on $X$.

\end{theorem}

Our goal is to construct a subsolution of the $J$-equation (\ref{j3d})  in an open neighborhood of any subvariety $Z$ of $X$.   Let  $\omega_0\in \alpha$ and $\chi\in \beta$ be two fixed K\"ahler forms on $X$. 

\begin{lemma}\label{s2l21} Let $Z$ be an analytic subvariety of $X$ with $\dim Z \leq 1$. Then there exists an open neighborhood $U$ of $Z$ and a smooth K\"ahler metric $\omega_U \in [\omega]|_U$ such that
$$3 (\omega_U)^2 - 2 \omega_U\wedge \chi>0$$
in $U$.

\end{lemma}

\begin{proof} If $\dim Z=0$, the lemma is obvious since every K\"ahler class in a sufficiently small neighborhood of a point is trivial.

If $\dim Z=1$, any two K\"ahler classes on $Z$ are proportional to each other and by the extension theorem (c.f. Proposition 3.3 \cite{DP}), there exist a neighborhood $U_Z$ of $Z$ and a K\"ahler form $\omega_Z\in \alpha|_{U_Z}$ such that
$$3\omega_Z - (1+2\epsilon) \chi>0$$
for some $\epsilon>0$.
Let $\cS_Z$ be the singular set of $Z$. Then there exist a sufficiently small $\delta>0$ and $\varphi \in C^\infty(U_Z \setminus \cS_Z)\cap \PSH(U_Z, \omega)$ such that
$$3(\omega_Z +\ddbar \varphi) - (1+\epsilon)\chi>0$$
and $\varphi$ has Lelong number greater than $\delta$ at $\cS_Z$. 

For simplicity, we assume $\cS_Z$ is a single point $p$ of $Z$ and components of $Z^\circ=Z\setminus \{p\}$ are smooth open curves in a neighborhood $U_p$ of $p$ in $X$. We assume there exists $\varphi_p \in C^\infty(U_p)$ with 
$$3(\omega_p)^2 > 2\omega_p\wedge \chi, ~\omega_p =  \omega_Z+ \ddbar \varphi_p $$
where $U_p$ is a sufficiently small neighborhood of $p$. By subtracting a sufficiently large number from $\varphi_p$, we can assume
$$\varphi_p < \varphi-2$$
on $U_p \setminus V_p$ for some sufficiently small open neighborhood $V_p\subset\subset U_p$ of $p$. Since $\varphi \rightarrow -\infty$ at $p$,  there exists a neighborhood $W_p \subset\subset V_p$ of $p$ such that
$$\varphi_p > \varphi +2$$
in $W_p$. Let $Z_1, Z_2, ..., Z_k$ be all the components of $Z$ in $U_p$. 
We replace $\varphi$ in $U_Z \setminus W_p$ by 
$$\tilde\varphi= \varphi+ A d^2,$$
where $d$ is the distance function to $Z_1, ..., Z_k$ with respect to $\omega_Z$. For sufficiently large $A>0$, $\tilde\omega_Z= \omega_Z+\ddbar \tilde\varphi$ is a K\"ahler form and 
$$3\tilde\omega^2 > (2+\epsilon) \tilde\omega \wedge \chi$$
in some sufficiently small open neighborhood $\tilde U$ of $(Z_1 \cup...\cup Z_k) \cap (U_p \setminus W_p)$.  We can take the regularized maximum of $\tilde\varphi$ and $\varphi_p$ in $ \tilde U$ and denote it by $\varphi_Z$. The lemma is then proved by choosing $\omega_U= \omega+ \ddbar \varphi_Z$ for some sufficiently small open neighborhood $U$ of $Z$. 

\end{proof}

Now let $Z$ be an analytic subvariety of $X$ with $\dim Z=2$. By the assumption of the  $J$-positive condition, 
\begin{equation}\label{surf1}
3\alpha - \beta>0
\end{equation}
is a K\"ahler class on $Z$ and 
\begin{equation}\label{surf2}
 (3\alpha^2 - 2\alpha \cdot \beta)\cdot Z'>0 
 \end{equation}
for any irreducible component $ Z'$ of $Z$. 
Let $Z'$ be an irreducible component of $Z$ and $\cS_Z$ be the singular set of $Z$.
Let $\Phi: X' \rightarrow X$ be the resolution of singularities for $Z$. We can assume the exceptional locus of $\Phi$ is a union of smooth curves and it coincides with $\Phi^{-1}(\cS_Z)$. We also let $\hat Z$ be the strict transform of $Z'$ by $\Phi$.

 For conveniences, we identify $\Phi^*\omega_0$ and $\Phi^*\chi$ with $\omega_0$ and $\chi$. We also identify $\Phi^*\alpha$ and $\Phi^*\beta$ with $\alpha$ and $\beta$.

 There exists an effective $\mathbb{Q}$-divisor $E$ whose support coincides with the exceptional locus of $\Phi$ such that $\alpha- [E]$ is a K\"ahler class on $X'$. Let $\sigma_E$ and $h_E$ be the defining section and hermitian metric on the line bundle associated to $[E]$ such that for any sufficiently small $\epsilon>0$
$$3\omega_0 -\chi + \epsilon \ddbar \log h_E$$
is K\"ahler on $\hat Z$. 
There exists  $\varphi_\epsilon \in C^\infty(\hat Z)$ such that
$$( 3\omega_0 -\chi +  \epsilon \ddbar \log h_E + 3\ddbar \varphi)^2 =  c_\epsilon\chi^2,$$
where the constant $c_\epsilon>0$ is given by 
$$c_\epsilon \beta^2= \left( (3\alpha- \beta) - \epsilon [E] \right)^2.$$
 By choosing sufficiently small $\epsilon>0$, we can assume that
$$c_\epsilon> 1$$
since 
$$(3\alpha - \beta)^2 -\beta^2= 3(3\alpha^2 - 2\beta\cdot \alpha)>0$$
 on $\hat Z$.
By letting 
\begin{equation}\label{ssf1}
\varphi_{\hat Z}=\varphi_\epsilon+ \log |\sigma_E|^2_{h_E}, ~\omega_{\hat Z}= \omega_0+ \ddbar \varphi_{\hat Z},
\end{equation} 
we have on $\hat Z \setminus \Phi^{-1}(\cS_Z)$, 
\begin{equation}\label{ssf2}
3\omega_{\hat Z}^2 -2 \omega_{\hat Z}\wedge \chi =\frac{(3\omega_{\hat Z} - \chi)^2 -  \chi^2}{3}= \frac{c_\epsilon-1}{3} \chi^2>0.
\end{equation}
Furthermore, $\omega_{\hat Z}$ has positive Lelong number along $E$.

\begin{lemma} \label{s2l22} Suppose $Z$ is an analytic subvariety of $X$ of  dimension $2$. Then there exists an open neighborhood $U$ of $Z$ and a smooth K\"ahler metric $\omega_U \in \alpha|_U$ such that
$$3 (\omega_U)^2 - 2 \omega_U\wedge \chi>0$$
on $U$.

\end{lemma}

\begin{proof}     Let $Z_1$, ..., $Z_k$ be  the irreducible components of $Z\setminus \cS_Z$. For simplicity, we assume $k=1$ and $\cS_Z$ is connected.  
By Lemma \ref{s2l21}, there exists an open neighborhood $V$ of $\cS_Z$ and $\omega_V=\omega+\ddbar \varphi_V \in \alpha |_V$ for some $\varphi_V\in C^\infty(V)\cap \PSH(V, \omega)$ such that  $\omega_V> \chi$ in $V$ and so 
$$3 (\omega_V)^2 - 2\omega_V\wedge \chi >0.$$

We will then extend $\varphi_{\hat Z}$ given in (\ref{ssf1})  using the square of the distance function as in the proof of Lemma \ref{s2l21}. The rest of the proof is similar to that of Lemma \ref{s2l21}. We let $\varphi_Z$ be   the regularized maximum of $\varphi_V$ and  the extension of $\varphi_{\hat Z} $. The lemma is proved by letting $\omega_Z =\omega+\ddbar \Phi(\varphi_Z)$ since $\omega_Z$ coincides with $\omega_V$ in some open neighborhood of $\cS_Z$.

\end{proof}

Now we can directly apply the result of Chen \cite{ChG} to prove Theorem \ref{s2nthmain}. The induction argument in \cite{ChG} is based on the assumption of Lemma \ref{s2l21} and Lemma \ref{s2l22} if the subvariety $Z$ is smooth, where the uniform $J$-positivity is used instead of the $J$-positivity as in Definition \ref{jamp}. As long as Lemma \ref{s2l21} and Lemma \ref{s2l22} hold, the argument in section 4 of \cite{ChG} immediately implies Theorem \ref{s2nthmain}.

\section{Set-up}

In this section, we will lay out the proof of Theorem \ref{main1}. For conveniences, we will focus on the case when $X$ is irreducible. The proof can be easily generalized for reducible situations and we will explain in \S 7.    The rest of this section together with  \S 4,  \S 5,  \S 6 and \S 7 are devoted to the proof of the following theorem. 

\begin{theorem}\label{s3main}   Let $X$ be an irreducible $n$-dimensional compact analytic subvariety of an ambient K\"ahler manifold $\cM$. Let $\alpha$ and $\beta$ be two K\"ahler classes on $X$  satisfying
\begin{equation}\label{s3cond1} 
\frac{  \alpha^{n-1}\cdot \beta \cdot X }{ \alpha^n \cdot X }     \leq 1.
\end{equation}
Suppose for any $m$-dimensional subvariety $Z$ of $X$ with $1\leq m \leq n-1$,
\begin{equation}\label{s3cond2}
\left( n \alpha^m - m \alpha^{m-1} \cdot \beta \right)\cdot Z >0. 
\end{equation}
Then for any smooth K\"ahler form $\chi\in \beta$ in an open neighborhood of $X$,    there exists a smooth K\"ahler form $\omega\in \alpha$ such that in some open neighborhood of $X \subset \cM$, we have 
\begin{equation}\label{s3cond3}
n \omega^{n-1} - (n-1) \omega^{n-2}\wedge \chi >0. 
\end{equation}

\end{theorem}

The following lemma is obvious by direct calculations. 

\begin{lemma}\label{s3connec} Suppose $(\alpha, \beta)$ are a pair of K\"ahler classes on $X$ satisfying the assumptions   (\ref{s3cond1}) and (\ref{s3cond2}) in Theorem \ref{s3main}.  Let $\alpha(t)=(1+t)\alpha$ for  $t\geq 0$. Then the pair $(\alpha(t), \beta)$ also satisfy  (\ref{s3cond1}) and (\ref{s3cond2}). Furthermore, for any smooth K\"ahler form $\chi\in \beta$, there exists sufficiently large $T>0$ and a K\"ahler form $\omega(T)\in \alpha(T)$ satisfying (\ref{s3cond3}). 

\end{lemma}

Lemma \ref{s3connec}  implies that for any K\"ahler class $\beta$ of $X$, the set $\cK_\beta$ defined on $X$ by
$$\cK_\beta=\{ \alpha~|~\alpha  \textnormal{~ is K\"ahler and } (\alpha, \beta)~\textnormal{satisfies (\ref{s3cond1}) and (\ref{s3cond2}) }\}$$
is connected.

Before we start our proof, we will first introduce some notations and elementary lemmas. The following notion introduced in \cite{ChG} can be used to verify a subsolution for the $J$-equation.

\begin{definition} For any $N\times N$ positive definitive hermitian matrix $A$ and $B$, 
\begin{equation}
\cP_B(A) = \max_{V \subset \mathbb{C}^N}   \tr_{A|_V}(B|_V) ,
\end{equation}
where $V$ is any hyperplane of $\mathbb{C}^N$. 

\end{definition}

It is obvious $\cP_B(A)$ is invariant under unitary transformation and so  we can always make $B$ an diagonal or identity matrix. If $B$ is the identity matrix and the eigenvalues of $A$ are given by $\lambda_1$, ..., $\lambda_N$, then
$$\cP_B (A) = \max_{k=1, ..., N} \left(  \sum_{j=1, j\neq k}^N \lambda_j^{-1}\right). $$
  
\begin{lemma}\label{s3con1} Let  $A$,  $B$ and $C$ be $N\times N$ positive definitive hermitian matrices. Then for any $t\in [0,1]$, we have
$$\cP_C (tA+(1-t)B) \leq  t \cP_C (A )  + (1-t) \cP_C ( B) . $$

\end{lemma}
\begin{proof} We can assume $A$ is identity matrix and $B= diag(\lambda_1, ..., \lambda_N)$. Let $W$ be the hyperplane of $\mathbb{C}^N$ maximizing $\cP_C(tA+ (1-t)B)$. Let $\mu_1, ..., \mu_N$ be the diagonal elements of $C$. We can assume $W$ is spanned by $e_1, ..., e_{N-1}$ and 
$$\cP_C (tA+ (1-t)B ) = \sum_{j=1}^{N-1} \mu_j (t + (1-t)\lambda_j)^{-1}.$$
By convexity of the function $x^{-1}$, we have 
\begin{eqnarray*}
\cP_C (tA+ (1-t)B ) &= &\sum_{j=1}^{N-1} \mu_j (t + (1-t)\lambda_j)^{-1}\\
&\leq & \sum_{j=1}^{N-1} \mu_j \left(t  + (1-t)(\lambda_j)^{-1}\right)\\
&=&t \cP_C(A) + (1-t) \cP_C(B).
\end{eqnarray*}

\end{proof}

Lemma \ref{s3con1} shows that $\cP_B(A)$ is convex in $B$. We can now extend the definition $\cP_B(A)$ to 
$$\cP_\chi( \omega )$$
for any two K\"ahler forms $\omega$ and $\chi$ on a K\"ahler manifold. The following lemma is due to Chen (\cite{ChG} Lemma 3.5).

\begin{lemma} \label{s3con2} Suppose $A$, $B$ and $C$ are $N\times N$ complex-valued matrices such that 
$\begin{bmatrix}
A & C \\
\overline{C}^T & B
\end{bmatrix}
$ is a positive definite hermitian matrix. 
Then
$$\cP_{I_N} (A - CB^{-1} \overline{C}^T) + \tr_B (I_N) \leq \cP_{I_{2N}} \left( \begin{bmatrix}
A & C \\
\overline{C}^T & B
\end{bmatrix}\right),  $$
where $I_N$ and $I_{2N}$ are identity matrices of rank $N$ and $2N$ respectively.
\end{lemma} 

The following lemma shows that the operator $\cP$ preserves the upper bound by taking the maximum for plurisubharmonic functions. 
 \begin{lemma} \label{s3l34} Let $u$ and $v$ be two smooth plurisubharmonic functions on an open domain $U\subset \C^n$ such that
 $$P_{\omega_{\C^n}} (\ddbar u) <c, ~P_{\omega_{\C^n}} (\ddbar v) <c$$
 for some $c>0$, where $\omega_{\C^n}$ is the standard Euclidean metric on $\C^n$. Then 
 $$\cP_{\omega_{\C^n}} (\ddbar \max(u, v)) <c$$
 in the sense that
 $$nc \left(\ddbar \max(u, v) \right)^{n-1} > (n-1) \left(\ddbar \max(u, v) \right)^{n-2} \wedge \omega_{\C^n}$$
 as currents.

 \end{lemma}

\begin{proof} Let $\omega=\ddbar \max(u, v)$ be a positive closed current.  Then
$$\omega=\ddbar \max(u, v) \geq \chi_{ \{u\geq v\}} \ddbar u + \chi_{\{u<v\}} \ddbar  v$$
in the sense of currents (c.f. \cite{Kol2}). Since $u$ and $v$ are smooth, we can define $\eta =\chi_{ \{u\geq v\}} \ddbar u + \chi_{\{u<v\}} \ddbar  v$ to be a bounded hermitian form. Then $\cP_{\omega_{\C^n}}(\eta) <c$. The local regularization of a positive closed current $\Theta$ is defined  by 
$$
\Theta^{(r)}(p) = \int_{B_r(0)} r^{-2n} \rho\left(\frac{|y|}{r}\right)  \Theta(p-y) d\textnormal{Vol}_{\C^n} (y)%
$$
where $r>0$ and  $\rho(t)$ is a smooth non-negative function with support in $[0,1]$ satisfying
$$\int_{B_1(0)}   \rho(|y|) d\textnormal{Vol}_{\C^n}(y)=1. $$ For anyt $V\subset\subset U$, there exists $\varepsilon>0$ such that  in $V$,
$$\cP_{\omega_{\C^n}}(\omega^{(r)}) \leq  \left( \cP_{\omega_{\C^n}} (\eta) \right)^{(r)} < c-\varepsilon$$
by convexity of the operator $\cP$,
or equivalently, 
$$n(c-\varepsilon)((\omega^{(r)} )^{n-1} > (n-1) (\omega^{(r)})^{n-2} \wedge \omega_{\C^n}, $$
for any sufficiently small $r>0$. The lemma is proved by letting $r\rightarrow 0$. 

\end{proof}

Lemma \ref{s3l34} also holds if the plurisubharmonic functions $u$ and $v$ are continuous. It also immediately implies that $\cP_{\omega_{\C^n}}(\ddbar\tilde\max(u, v))< c$, where $\tilde\max(u, v)$ is the regularized maximum of $u$ and $v$.

We will prove Theorem \ref{s3main} by induction on dimension $m=\max_i \dim Z_i$ of the subvariety $Z$ of $X$, where $Z_i$ are components of $Z$. 

\begin{lemma}\label{s3lowd} Under the assumption of Theorem \ref{s3main},  for any subvariety $Z$ of $X$ with $\dim Z\leq 1$, there exist an open neighborhood $U$ of $Z$ in $\cM$ and  a K\"ahler form $\omega_{Z} \in \alpha|_{U}$ such that in $U$, 
$$n (\omega_{Z})^{n-1} - (n-1) (\omega_{Z})^{n-2}\wedge \chi >0. $$

\end{lemma}

\begin{proof} The proof is identical to that of Lemma \ref{s2l21}.

\end{proof}

\section{The twisted $J$-equation}

We first state the following  PDE theorem due to Chen (Theorem 1.14 \cite{ChG}) for the twisted $J$-equation. 

\begin{theorem} \label{chg} Let $M$ be an $n$-dimensional compact K\"ahler manifold equipped with two K\"ahler forms $\omega_0$ and $\chi$. Suppose $c\in \mathbb{R}^+$ and  $F\in C^\infty(M)$ satisfy 
$$ F> - \frac{1}{2n^{n+1} c^{n-1}}, ~ \int_M F \chi^n= c \int_M \omega_0 ^n -  \int_M \chi \wedge \omega_0^{n-1}  >0.$$
If 
$$nc\omega_0^{n-1} - (n-1) \chi\wedge \omega_0^{n-2} >0,$$
then there exists a unique K\"ahler form $\omega \in \omega_0$ satisfying the twisted $J$-equation
\begin{equation}\label{twJ}
c\omega^n = \chi \wedge \omega^{n-1} + F\chi^n
\end{equation}
and the inequality
$$nc\omega^{n-1} - (n-1)\chi \wedge \omega^{n-2}>0.$$

\end{theorem}

Equation (\ref{twJ}) coincides with the $J$-equation if we choose $F=0$ and $c=1$ with normalization $\alpha^n = \alpha^{n-1}\cdot \beta$. The advantage of Theorem (\ref{twJ}) is that  $F$ is allowed to be negative.

We continue our proof of Theorem \ref{s3main} by induction of dimension. We will fix the K\"ahler metric 
\begin{equation}
\chi\in \beta, ~\omega_0\in \alpha.
\end{equation}
in some open neighborhood of $X$ from the assumption of Theorem \ref{s3main}. Without loss of generality, we can assume that
\begin{equation}\label{kah}
\lambda^{-1}\chi < \omega_0 < \lambda \chi
\end{equation}
for some sufficiently large $\lambda \geq 1$. 

By induction, we can assume that for any subvariety $Z$ of $X$ with  $ \dim Z \leq m-1<n-1$, there exists a smooth K\"ahler form $\omega_Z \in \alpha$ such that in some open neighborhood of $Z$, we have 
$$ n (\omega_Z)^{n-1} - (n-1) (\omega_Z)^{n-2}\wedge \chi >0$$
or
$$\cP_\chi (\omega_Z) <n . $$

Now we let $Z$ be an $m$-dimensional analytic subariety of $X$. We apply the resolution of singularities for  $Z$ by 
$$\Phi: \cM' \rightarrow \cM$$
such that $Z'$, the strict transform of all $m$-dimensional components $Z$ by $\Phi$, is a disjoint union of smooth $m$-dimensional submanifolds $\cM'$. In fact, we can choose $\Phi$ to be successive blow-ups along smooth centers.  For conveniences, we still use $\alpha$ and $\beta$ for $\Phi^*\alpha$ and $\Phi^*\beta$. Obviously,  the pair $(\alpha, \beta)$ is $J$-nef  on $X'$, although they are big and semi-positive.

We will perturb the K\"ahler class $\alpha$ to obtain strict positivity of $(\alpha, \beta)$ on $Z'$.  Let $\theta$ be a K\"ahler form on $\cM'$ and so $\theta$ restricted to $X'$ or $Z'$ is also a K\"ahler form.  We let 
\begin{equation}\label{hatal}
\hat\alpha=\hat\alpha(t, \epsilon) = (1+2\lambda t) \alpha + \epsilon t [\theta], ~~~\hat\beta=\hat\beta(t, \epsilon)= (1+t)\beta + \epsilon^2 t \min(1, t) [\theta].
\end{equation}

\begin{lemma} \label{320}  There exists $\epsilon_Z>0$ such that for any  $t>0$, $0< \epsilon< \epsilon_Z$ and any $k$-dimensional subvariety $V'$ of $Z'$ with $1\leq k \leq m$, we have  
\begin{equation}
\left(   n  \hat \alpha^k - k \hat \alpha ^{k -1}\cdot\hat\beta  \right) \cdot  V' > 0, 
\end{equation}
where $\hat\alpha=\hat\beta=\hat\alpha(t, \epsilon)$ and $\hat\beta=\hat\beta(t, \epsilon)$ are defined as in (\ref{hatal}).
\end{lemma}

\begin{proof} We first verify the case when $k=m$. Straightforward calculations show that there exists $C=C(Z, \alpha, \beta)>0$ such that for any $t>0$
\begin{eqnarray*}
&& \left( n  \hat\alpha  ^m - m  \hat\alpha ^{m -1}\cdot \hat \beta\right) \cdot Z' \\
&=& \left(n(1+2\lambda t)\alpha - m(1+t) \beta+ \epsilon t \left(n-  \epsilon m \min(1, t) \right)\theta  \right)  \cdot \hat\alpha^{m-1} \cdot Z' \\
&\geq&(1+t)^{m-1} \left(  n \alpha^m  - m \alpha^{m-1}\cdot \beta\right) \cdot Z'  - C\sum_{k=1}^{m-2} (1+t)^k(\epsilon t)^{m-1-k} \alpha^k \cdot [\theta]^{m-k-1}\cdot Z' \\
&>& 2^{-1}   \left(  n \alpha^m  - m \alpha^{m-1}\cdot \beta\right) \cdot Z' \\
&>& 0.
 \end{eqnarray*}
by choosing sufficiently small $\epsilon=\epsilon(Z, \alpha, \beta)>0$.

For $1\leq k < m$,  we can assume $V'$ is irreducible and let $W = \Phi(V')$. Then $W$ is a subvariety of $X$ of $\dim W \leq k <m$.  By the assumption of the induction argument, we can assume that there exists a K\"ahler form $\omega_W \in \alpha $  in an open neighborhood of $W$ in $\cM$ satisfying
$$ n (\omega_W)^{n-1} - (n-1) (\omega_W)^{n-2} \wedge \chi >0 .$$
It immediately implies that for any $l \leq n-1$, 
$$ n (\omega_W)^l - l (\omega_W)^{l-1} \wedge \chi >0. $$
We let $\omega_{V'}=\Phi^*\omega_W$. Then %

$$ \left( n (\omega_{V'})^l - l (\omega_{V'})^{l-1} \wedge \chi \right) \wedge \theta^{k-l}\geq0 $$
in $V'$ for $l=1, 2, ..., k$.  
We let 
$$\eta(t, \epsilon) = (1+2\lambda t) \omega_{V'} + \epsilon t \theta, ~ \hat\chi(t, \epsilon) = (1+t)\chi+ \epsilon^2 t \min(1,t) \theta. $$
For conveniences, we write $\hat\alpha$, $\hat \beta$, $\eta$ and $\hat\chi$ for $\hat \alpha(t, \epsilon)$, $\hat \beta(t, \epsilon)$, $\eta(t, \epsilon)$ and $\chi(t, \epsilon)$. Then for  $t>0$ and sufficiently small $\epsilon>0$, we have 
\begin{eqnarray*}
&&\left(   n \hat \alpha ^k - k \hat \alpha  ^{k-1}\cdot \hat \beta\ \right) \cdot V' \\
&=& \int_{V'} \left( n \eta ^k - k \eta ^{k-1}\wedge \hat\chi\right)\\
&=& \sum_{l=0}^k n \begin{pmatrix} k\\ l \end{pmatrix}(1+2\lambda t)^l (\epsilon t) ^{k-l}  \int_{V'} (\omega_{V'})^l \wedge \theta^{k-l} \\
&&- \sum_{l=1}^k  k  \begin{pmatrix} k-1\\ l-1 \end{pmatrix}(1+t) (1+2\lambda t)^{l-1} (\epsilon t)^{k-l} \int_{V'} (\omega_{V'})^{l-1} \wedge \chi \wedge \theta^{k-l} \\
&& - \sum_{l=0}^{k-1} k \begin{pmatrix} k-1\\ l \end{pmatrix}(1+2\lambda t)^l  (\epsilon t)^{k-l} \epsilon \min(1, t) \int_{V'}  (\omega_{V'})^l \wedge \theta^{k-l}\\
&= & \sum_{l=1}^{k} \begin{pmatrix} k\\ l \end{pmatrix} (1+2\lambda t)^{l-1} (\epsilon t) ^{k-l} \left(  n (1+2\lambda t) - n(1+t) - (k-l) (1+2\lambda t) \epsilon \min(1,t)  \right) \int_{V'}   (\omega_{V'})^l \wedge\theta^{k-l} \\
&&+ \sum_{l=1}^{k} \begin{pmatrix} k\\ l \end{pmatrix} (1+t) (1+2\lambda t)^{l-1} (\epsilon t) ^{k-l}  \int_{V'}   \left( n \omega_{V'} - l \chi \right) \wedge(\omega_{V'})^{l-1} \wedge\theta^{k-l} \\
&&+ (\epsilon t)^k \left( n - k \epsilon \min (1,t)  \right) \int_{V'} \theta^k    \\
&>&0
\end{eqnarray*}
 by the choice of $\eta$ and by choosing sufficiently small $\epsilon=\epsilon(k, \lambda)>0$.

\end{proof}

Since $Z'$ is a union of disjoint smooth submanifolds of $\cM'$ (or $X'$), for convenience, we let $\hat Z$ be a fixed component of $Z'$. We let $S_Z$ be the singular set of $Z$ and $\cS_{\hat Z}= \hat Z \cap \Phi^{-1}(\cS_Z)$. Then $\cS_{\hat Z}$ coincides with  the exceptional locus of $\Phi$ on $\hat Z$. We will consider the following twisted $J$-equation on $\hat Z$
\begin{equation} \label{332}
n (\hat \omega)^m = m  (\hat \omega )^{m-1}\wedge \hat \chi  + c_{t, \epsilon} ( \hat \chi)^m, 
\end{equation}
for $t>0$ and $\epsilon \in (0,\epsilon_Z)$, where 
$\epsilon_Z>0$ is defined in Lemma \ref{320}, $$\hat \omega= \hat\omega(t, \epsilon) \in \hat \alpha=\hat\alpha(t, \epsilon)$$ is a K\"ahler form and $c_{t, \epsilon}$ is the normalization constant defined by
\begin{equation} \label{coneqn}
 n (\hat \alpha )^m = m (\hat \alpha)^{m-1} \cdot \hat \beta + c_{t, \epsilon} (\hat \beta)^m. 
 \end{equation}
By Lemma \ref{320}, for any $\epsilon\in(0,\epsilon_Z)$ and $t>0$, we have $$c_{t, \epsilon}>0. $$

Equation (\ref{332})) can always be solved for  $t=1$ for any $\epsilon \in (0, \epsilon_Z)$.  In fact, if we let $\eta=\eta(t, \epsilon) = (1+2\lambda t)\omega_0+\epsilon t \theta$, then at $t=1$, 
\begin{eqnarray*}
&&n\eta^{m-1} - (m-1) \eta^{m-2} \wedge\hat\chi \\
&=& \left((1+2\lambda ) \omega_0 - 2\chi + \epsilon (n -(m-1)\epsilon)   \theta \right)\wedge \eta^{m-2}\\
&>&0
\end{eqnarray*}
by the assumption (\ref{kah}). Therefore  we can apply  Theorem \ref{chg} to solve equation (\ref{332}).  Let 
$$\cT_\epsilon = \{ t\in (0, 1] ~|~(\ref{332}) \textnormal { has a smooth solution at } t \textnormal{ for } \epsilon \in (0, \epsilon_Z) \}$$
and $$t_\epsilon=\inf \cT_\epsilon.$$

Obviously, $\cT_\epsilon$ is open and by applying the continuity method, for each $\epsilon\in (0, \epsilon_Z)$, equation (\ref{332}) can be solved for all  $t\in (t_\epsilon, 1]$.

\begin{lemma} \label{s4cont} For any $\epsilon\in (0,\epsilon_Z)$, $t_\epsilon=0$.

\end{lemma}

 We will assume Lemma \ref{s4cont} in \S 5 and \S 6.  The proof of  Lemma \ref{s4cont} will be given in \S 7, by going through the same argument  in \S 5 and \S 6. Indeed if $t_\epsilon>0$ for some $\epsilon\in (0, \epsilon_Z)$, equation (\ref{332}) is a non-degenerate at $t_\epsilon$ and one can repeat the same argument for $t_\epsilon=0$ to obtain a smooth subsolution for (\ref{332}) and then a smooth solution.

\section{A  mass concentration in the degenerate case}

In this section, we will prove a mass concentration result similarly as in \cite{ChG} based on the techniques in \cite{DP}. We will keep the notations as in \S 4 and assume Lemma \ref{s4cont}.   

Before stating the main result of this section. We will define a local regularization for global positive currents.

\begin{definition} \label{s5regdef} Let $M$ be an $n$-dimensional K\"ahler manifold and $\Theta$ be a closed $(1,1)$-positive current. Then the local regularization $\Theta^{(r)}=\{ \Theta_j^{(r)}\}_{j\in \cJ}$ of $\Theta$ with respect to a finite partition $\{B_j\}_{j\in \cJ} $ of $M$ and a scale $r$ is defined as follows. 
\begin{enumerate}

\medskip

\item $\{ B_j\}_j$ is a finite open covering of $M$. Each $B_j$ is biholomorphic to a Euclidean unit ball $B_1(0)$ in $\mathbb{C}^n$ equipped with a standard Euclidean metric $g_j$.  We also require that $\{2B_j\}_j$ is also a covering of $M$, where $2B_j$ is biholomorphic to $B_2(0)$ with respect to $g_j$. 

\medskip

\item $\Theta_j^{(r)}(x)$ is the standard regularization in $B_j$ for $x\in B_j$, defined by the following convolution
\begin{equation}\label{locreg}
\Theta_j^{(r)}(x) = \int_{B_j} r^{-2n} \rho\left(\frac{|y|}{r}\right)  \Theta(x-y) d\textnormal{Vol}_{g_j} (y)
\end{equation} 
for $r\in (0, 1)$, where $\rho(t)$ is a smooth non-negative function with support in $[0,1]$ satisfying
$$\int_{B_1(0)}   \rho(|y|) d\textnormal{Vol}_{g_j}(y)=1. $$

\end{enumerate}

\end{definition}

For simplicity we can assume that $\rho(t)$ is a decreasing function in $t$ and compactly supported on $[0,1]$. Also it  is constant on $[0, 1/2]$. The main advantage of applying the local regularization is that it does not lose positivity while the global regularization might. The other advantage is that if $\Theta = \ddbar \varphi$ locally for some plurisubharmonic function $\varphi$, then $\Theta^{(r)} = \ddbar \varphi^{(r)}$, where $\varphi^{(r)}$ is the local regularization defined by the same way as in (\ref{locreg}).

Now we can state the main result of this section. First, we have to pick a partition for $\hat Z$. Unlike the nondegenerate case in \cite{ChG}, $\hat\chi$ become degenerate near the exceptional locus of $\Phi$ as $t, \epsilon \rightarrow 0$. Even after local regularization, the positivity can not be maintained uniformly as it depends on the scale of regularization $r$.  For any small $\hat\epsilon>0$, we can choose a sufficiently fine open covering $\{B_j\}_{j\in \cJ}$ equipped a local Euclidean metric $g_j$ as in Definition \ref{s5regdef} so that in each $B_j$, there exists K\"ahler metric $\hat\chi_j$ with  constant coefficients such that in each $2B_j$
\begin{equation}\label{s5epchoi}
\hat \chi_j \leq \chi  \leq \hat\chi_j + \hat\epsilon g_j
\end{equation}
for some fixed small $\hat\epsilon>0$ to be chosen later. In fact, for any given $\varepsilon>0$, by picking sufficiently small $t$ and $\epsilon>0$ dependent on $\varepsilon$, we will choose 
$%
\hat\epsilon=\hat\epsilon(\varepsilon)< (2m)^{-1} \delta_0\varepsilon, 
$
where $\delta_0$ is defined as in Lemma \ref{s5mass}. 

We can always assume $g_j$  is quasi-equivalent to a fixed K\"ahler metric on $\hat Z$ independent of the choice $\{B_j\}_j$ because one can fix a partition and make it into a finer partition.  For convenience, let $\theta$ be a fixed K\"ahler metric on $\hat Z$ and we assume that in each $2B_j$
\begin{equation} \label{thetacho}
 2^{-1/100} g_j\leq  \theta \leq 2^{1/100} g_j 
 \end{equation}
by choosing finer partitions. We define for $t\in (0, 1)$ and $\epsilon\in (0, \epsilon_Z)$
\begin{equation}
\hat\chi=\chi(t, \epsilon) = (1+t) \chi +\epsilon^2 t\min(1,t) \theta \in \hat \alpha, ~\hat\omega_0= (1+2\lambda t)\omega_0 + \epsilon t \theta \in \beta, 
\end{equation}
where $\epsilon_Z$ is given in Lemma \ref{320} and $\lambda$ in (\ref{kah}).

 In fact, we will choose $\theta$ as in (\ref{s5phi}). We also assume that for any $t\in (0, \epsilon_Z)$ and $\epsilon\in (0,\epsilon_Z)$, 
\begin{equation}\label{s5kdef}
 \hat\chi \leq \lambda ~\hat \omega_0
\end{equation}
for some fixed $\lambda>1$.

\begin{theorem} \label{masscon}There exist $\delta>0$, a finite Euclidean partition $\{B_j\}_{j\in \cJ} $ of $\hat Z$, $\varepsilon>0$,  $r_0>0$ and  a K\"ahler current $\Omega \in (1-\delta)  \alpha $,  such that for all $0<r<r_0$, $j\in \cJ$, we have in $B_j$, 
$$\cP_{ \chi } \left( \Omega^{(r)} \right) < n- \varepsilon. $$
Furthermore, $\Omega$ has positive Lelong number along $\cS_{\hat Z}$.

\end{theorem}

The rest of the section is devoted to the proof of Theorem \ref{masscon} using ideas from \cite{ChG, DP}. We let 
$$\cZ = \hat Z\times \hat Z$$
 and 
 $$\Delta =\{ (p, p) \in \cZ~|~p\in \cZ\}$$
  be the diagonal submanifold of $\cZ$. 
Suppose $\Delta $ is locally defined by holomorphic functions $\{ f_{j, k}\}_{j, k}$ on finitely many domains $\{U_j\}_j$ covering $\cZ$. We define
$$\psi_s = \log \left( \sum_j \rho_j \sum_k |f_{j, k}|^2 +s\right)$$ 
 for $s\in (0, \epsilon_Z)$, where $\{\rho_j\}_j$ is a partition of unity for $\{U_j\}_j$. 
 
 Let $\pi_1$ and $\pi_2$ be the projections maps from $\cZ$ to $\hat Z$ and let
 $$\chi_{\cZ} = \chi_{\cZ}(t, \epsilon)= \pi_1^*~\hat\chi + \pi_2^*~ \hat \chi,$$
 for $t, \epsilon>0$, where $\hat\chi= \hat\chi(t, \epsilon)$ is defined as in \S 4.  

Since we assume the resolution $\Phi$ is a successive blow-up along smooth center, there exists $\phi\in \textnormal{PSH}(\hat Z, \chi)$ such that $\phi$ is smooth away from $\cS_{\hat Z}$ and on $\hat Z\setminus \cS_{\hat Z} $, We can choose the K\"ahler form $\theta$ (as in (\ref{thetacho}) ) on $\hat Z$ such that 
\begin{equation}\label{s5phi}
\theta= \chi + \ddbar \phi>0
\end{equation}
 on $\hat Z \setminus \cS_{\hat Z}$. We can also assume that $\omega_0 +\ddbar \phi$ is also a K\"ahler form on $\hat Z \setminus \cS_{\hat Z}$ satisfying
 $$  \lambda^{-1} \theta \leq \omega_0 + \ddbar \phi \leq  \lambda \theta $$
on $\hat Z \setminus \cS_{\hat Z}$ for sufficiently large $\lambda \geq 1$.  In particular, $\phi$ has positive Lelong number along $\cS_{\hat Z}$.  If we let 
 $$\tilde \phi = \pi_1^*\phi+\pi_2^*\phi, $$ 
 then $(\chi_\cZ + \ddbar \tilde \phi)|_{\cZ \setminus \left( \pi_1^{-1}(\cS_{\hat Z}) \cup \pi_2^{-1}(\cS_{\hat Z}) \right)}$ extends to a K\"ahler metric on $\cZ$. We also define 
 \begin{equation}\label{permass}
 \chi_{\cZ, s} =\chi_{\cZ, s}(t, \epsilon)= \chi_\cZ + s \ddbar \tilde\phi+ s^2 \ddbar \psi_s,
 \end{equation}
for a fixed sufficiently small $s>0$ so that $\chi_{\cZ, s}$ is a K\"ahler current and
 \begin{equation}\label{fdef}
 F_{\cZ,  s} =F_{\cZ,  s} (t, \epsilon) = \frac{ (\chi_{\cZ, s})^{2m}} {(\chi_{\cZ})^{2m}} - (1+c_{t, \epsilon, s}) +  \frac{c_{t, \epsilon}   }{m+n}, 
 \end{equation}
 where $c_{t, \epsilon}$ is defined in equation (\ref{coneqn}) and $c_{t, \epsilon, s}$ is defined by the normalization
 $$\int_{\cZ}(\chi_{\cZ, s})^{2m}= (1+c_{t, \epsilon, s}) \int_{\cZ}(\chi_{\cZ})^{2m}.$$

Note that $\lim_{s\rightarrow 0} c_{t, \epsilon, s} = 0$.   We then consider the following twisted $J$-equation on $\cZ$ by
 \begin{equation}\label{doubleeqn}
 (\omega_{\cZ, s})^{2m} = (\omega_{\cZ, s})^{2m-1}\wedge \left( \frac{2m}{m+n}\right) \chi_{\cZ} + F_{\cZ, s} (\chi_{\cZ})^{2m},  
 \end{equation}
 where
\begin{equation}
\omega_{\cZ, s} \in \pi_1^* ~ \hat \alpha  +  \pi_2^* ~\hat \beta   
\end{equation}
 and 
  \begin{equation}
 \int_\cZ F_{\cZ, s} (\chi_{\cZ})^{2m} =  \frac{c_{t,\epsilon} }{m+n} \int_\cZ (\chi_\cZ)^{2m}>0
 \end{equation}
  Equation (\ref{doubleeqn}) can also be written as 
 \begin{equation} \label{doubleeqn2}
 tr_{\omega_{\cZ, s}} (\chi_{\cZ}) +  (m+n) F_{\cZ, s}    \frac{ (\chi_\cZ)^{2m}}{ (\omega_{\cZ, s})^{2m}}  =  m+n.
 \end{equation}

  \begin{lemma}\label{flb} There exists a sufficiently small $s>0$, such that for any $t\in (0, 1)$, $\epsilon\in (0,\epsilon_Z)$ and $s\in (0, 1)$,
 $$\inf_\cZ F_{\cZ, s} > - \frac{1}{2 (2m)^{2m+1} (m+n)^{2m}}.$$ 
 
 \end{lemma}
 
 \begin{proof} Direct local calculations show that there exists $A>0$ such that for $s\in (0, \epsilon_Z)$.
 $$A \theta+ \ddbar \log \psi_s >0. $$
 Therefore by definition of $\tilde \phi$ and $\chi_{\cZ, s}$, for any $\varepsilon>0$, there exists sufficiently small $s>0$ such that for any $s\in (0, \epsilon_Z)$, $t\in (0, \epsilon_Z)$ and $\epsilon\in (0, \epsilon_Z)$
 $$\chi_{\cZ, s} \geq (1-\varepsilon) \chi_\cZ$$
 and
 $$F_{\cZ, s} \geq (1-\varepsilon)^n -1. $$
 The lemma then immediately follows by choosing sufficiently small $\varepsilon>0$. 
 
 \end{proof}
 
 We will fix $s>0$ once for all from Lemma \ref{flb}. 
 
 \begin{lemma} For any  $s\in (0, 1)$, $\epsilon\in (0, \epsilon_Z)$ and $t\in (0,1)$, there exists a unique smooth K\"ahler form $\omega_{\cZ, s}=\omega_{\cZ, s}(t, \epsilon)$ that solves equation (\ref{doubleeqn}) and satisfies
 \begin{equation} \label{doubleeqn3}
(m+n) (\omega_{\cZ, s})^{2m-1} - (2m-1) (\omega_{\cZ, s})^{2m-2} \wedge  \chi_\cZ>0.
 \end{equation}

 \end{lemma}

\begin{proof} First, we note that equation (\ref{doubleeqn}) is well-defined because
 $$\int_\cZ (\omega_{\cZ, s} )^{2m} = \begin{pmatrix} 2m\\ m \end{pmatrix}  \int_{\hat Z} \hat\alpha^m  \int_{\hat Z} \hat\beta^m $$
 and by (\ref{coneqn}), 
 \begin{eqnarray*}
 && \int_\cZ \left( (\omega_{\cZ, s})^{2m-1}\wedge \left( \frac{2m}{m+n}\right) \chi_{\cZ} + F_{\cZ, s} (\chi_{\cZ})^{2m}\right) \\
 &=& \frac{2m}{m+n} \begin{pmatrix} 2m-1\\ m \end{pmatrix} \left( \int_{\hat Z} \hat \alpha^m + \int_{\hat Z}\hat\alpha^{m-1}\cdot \hat\beta\right) \int_{\hat Z}\hat\beta^m + \frac{ c_{t,\epsilon} }{m+n} \begin{pmatrix} 2m\\ m \end{pmatrix} \left( \int_{\hat Z} \hat \beta^m\right)^2 \\
&=& (m+n)^{-1} \begin{pmatrix} 2m\\ m \end{pmatrix}  \left( m \int_{\hat Z} \hat \alpha^m + m \int_{\hat Z}\hat\alpha^{m-1}\cdot \hat\beta +  c_{t, \epsilon} \int_{\hat Z} \hat \beta^m\right) \int_{\hat Z}\hat\beta^m  \\
&=& \begin{pmatrix} 2m\\ m \end{pmatrix} \int_{\hat Z} \hat\alpha^m  \int_{\hat Z} \hat\beta^m.
 \end{eqnarray*}

Since we assume Lemma \ref{s4cont}, equation (\ref{332}) can be solved for all $t\in (0,1)$ and  $\epsilon \in (0, \epsilon_Z)$. We let $\hat\omega$ be the solution of (\ref{332}) for $t\in (0,1)$ and $\epsilon \in (0, \epsilon_Z)$. Then on $\hat Z$, we have
$$n \hat\omega^{m-1} - (m-1) \hat\omega^{m-2}\wedge \hat\chi>0$$
and
$$tr_{\hat\omega}(\hat\chi) < n. $$
We define 
$$\hat\omega_\cZ = \pi_1^* ~\hat\omega + \pi_2^*~ \hat\chi. $$
Then
\begin{eqnarray*}
&& \cP_{\hat\omega_\cZ} (\chi_\cZ) \\
&\leq& \max\left(  \cP_{\hat\omega }(\hat\chi)|_{\hat Z} + tr_{\hat\chi}(\hat\chi)|_{\hat Z}, tr_{\hat\omega}(\hat\chi)|_{\hat Z} +m-1\right) \\\
&<& n+m
\end{eqnarray*}
or equivalently
\begin{equation}\label{subforz}
(m+n) (\hat\omega_\cZ)^{2m-1} - (2m-1) (\hat\omega_\cZ)^{2m-2}\wedge\chi_\cZ>0.
\end{equation}
Therefore $\hat\omega_\cZ$ is a subsolution for equation (\ref{doubleeqn}) and we can now directly apply Theorem \ref{chg} for the twisted $J$-equation (\ref{doubleeqn})  by combining Lemma \ref{flb} and (\ref{subforz}) to obtain a unique solution of equation (\ref{doubleeqn}) satisfying (\ref{doubleeqn3}). This completes the proof of the lemma.

 \end{proof}

 We define $\omega_{s}=\omega_s(t, \epsilon)$ as the push-forward of $(\omega_{\cZ, s})^m \wedge \pi_2^* \hat\chi$ by $\pi_1$ as the following for any $z\in \hat Z$ 
 \begin{eqnarray}
 \omega_{s}(z)  &=& \cV^{-1} (\pi_1)_* \left( (\omega_{\cZ, s})^m \wedge \pi_2^*  \hat\chi  \right) (z)\\
 &=& \cV^{-1}   \int_{\pi_1^{-1}(z)} (\omega_{\cZ, s})^m \wedge \pi_2^*  ~\hat\chi   \nonumber,
 \end{eqnarray}
  where $$\cV =\cV (t, \epsilon)=  m  \hat\beta^m \cdot \hat Z.$$

  \begin{lemma} For any $s\in(0, 1)$, $t\in (0,1)$ and $\epsilon\in (0, \epsilon_Z)$, 
  $$\omega_{s } \in \hat\alpha .$$

  \end{lemma}
  
  \begin{proof} It suffices to calculate the push-forward of the K\"ahler class as follows 
  \begin{eqnarray*}
 \omega_s &=& \cV^{-1} \int_{\pi_1^{-1}(z)} (\omega_{\cZ, s})^m \wedge \pi_2^*   \hat\chi  \\
&\in& \cV^{-1}  \int_{\pi_1^{-1}(z)} ( \pi_1^* \hat\alpha +  \pi_2^* \hat\beta )^m \wedge \pi_2^*\hat\beta   =    \hat\alpha.
  \end{eqnarray*}

  \end{proof}

 Since $ (\omega_{\cZ, s})^m$ is a positive closed $(m, m)$-form, $\omega_s$ is a K\"ahler form on $\hat Z$ and $(\omega_s)^m$ always converges weakly as $s \rightarrow 0$, after passing to a sequence. We then define 
 \begin{equation}
 \Theta=\Theta(t, \epsilon)=  \lim_{ s\rightarrow 0} (\omega_{\cZ, s})^m
 \end{equation}
for the corresponding convergent subsequence.

\begin{lemma} \label{s5mass} There exists $\delta_0>0$ such that for any $t\in (0,1)$ and $\epsilon\in (0, \epsilon_Z)$, 
$$\Theta > 100\cV \delta_0 [\Delta],$$
where $[\Delta]$ is the current of integration along $\Delta$.
\end{lemma}
\begin{proof} Since $\omega_{\cZ, s}$ is the solution of (\ref{doubleeqn}), 
$$(n+m) \omega_{\cZ, s} > \chi_Z$$
by  Theorem \ref{chg}.  Then for some fixed sufficiently small $s>0$
\begin{eqnarray*}
 (\omega_{\cZ, s})^{2m} & \geq &\frac{2m}{(m+n)^{2m}} \chi_Z^{2m} +(\chi_{\cZ, s})^{2m} - (1+c_{t, \epsilon, s}- c_{t, \epsilon}) (\chi_\cZ)^{2m}\\
 &\geq & \frac{m}{(m+n)^{2m}} (\chi_{\cZ, s})^{2m}
 \end{eqnarray*}
 since $\lim_{s\rightarrow 0} c_{t, \epsilon, s}=0$. 
$\chi_\cZ$ is bounded above by a multiple of $\chi_{\cZ, s}$. 
The proof follows by the same argument in the proof of Proposition 2.6 in \cite{DP} by the fact that

\end{proof}

We remark that although $\Theta$ might depend on the choice of sequence $s_j$, its lower bound is independent of such a sequence and $t\in (0,1)$, $\epsilon\in(0, \epsilon_Z)$. 

Locally we let $z=(z_1, ..., z_m)$ be the coordinates for $\hat Z$ from $\pi_1$ and $w=(w_1, ..., w_m)$ for $\hat Z$ from $\pi_2$. Then we write
\begin{equation}
\omega_{\cZ, s} = \omega_H + \omega_M + \omega_{\overline M^T} + \omega_V,
\end{equation}
where $\omega_H$ is the horizontal component, i.e., in $\sqrt{-1} dz_i\wedge d\bar z_j$, $\omega_V$ is the vertical component in $\sqrt{-1} dw_i\wedge d\bar w_j$, $\omega_M$ and $\omega_{\overline M^T}$ are the off-diagonal or the mixed components in $\sqrt{-1} dz_i\wedge d\bar w_j$ or $\sqrt{-1} dw_i\wedge d\bar z_j$. We also let $\chi_H= \pi_1^* \hat\chi$ and $\chi_V= \pi_2^* \hat\chi$.
Our goal is to calculate
$$ \cP_{\hat\chi} (\omega_{s}).$$

\begin{lemma}  \label{s5l55}
At any point $z\in \hat Z$, for any $t\in (0,1)$, $\epsilon\in (0, \epsilon_Z)$ and $s\in (0, 1)$,  we have 
\begin{equation} 
\cP_{ \hat \chi } (\omega_{s}) < n .  
\end{equation}

\end{lemma}

\begin{proof} 

Direct calculations show that
\begin{eqnarray*}
 \omega_{ s}
&=& \cV^{-1}  (\pi_1)_* \left( (\omega_H + \omega_V + \omega_M + \omega_{\overline M^T} )^m \wedge \chi_V \right) \\
&=& \cV^{-1} (\pi_1)_* \left( m\omega_H \wedge \omega_V^{m-1} \wedge \chi_V + m(m-1)\omega_M \wedge \omega_{\overline{M}^T}\wedge \omega_V^{m-2} \wedge \chi_V  \right)\\
&= & \cV^{-1} (\pi_1)_* \left\{  \left( tr_{\omega_V}(\chi_V) \omega_H + tr_{\omega_V} \left(\chi_V  \wedge \omega_M \wedge \omega_{\overline{M}^T}\right) \right)\wedge \omega_V^m   \right\}. 
\end{eqnarray*}
At any point $p\in \pi_1^{-1}(z)$ for $z\in \hat Z$, we can assume 
$$\chi_V= \sqrt{-1}\sum_{i, j=1}^m  \delta_{ij}  dw_i\wedge d\bar w_j, ~\omega_V = \sqrt{-1} \sum_{i, j=1}^m \delta_{ij} \lambda_j dw_i\wedge d\bar w_j,$$  
$$\chi_H= \sqrt{-1}\sum_{i, j=1}^m  \delta_{ij}  dz_i\wedge d\bar z_j, ~\omega_H = \sqrt{-1} \sum_{i, j=1}^m \delta_{ij}   \mu_j dz_i \wedge d\bar z_j $$
and
$$\omega_M = \sqrt{-1} \sum_{i, j=1}^m \eta_{i\bar j} dz_i \wedge d\bar w_j,$$
where $(\delta_{ij})$ is the identity matrix. Then  $(1,1)$-form 
\begin{equation}\label{s5gamma} 
\Gamma = tr_{\omega_V}(\chi_V) \omega_H + tr_{\omega_V} \left(\chi_V  \wedge \omega_M \wedge \omega_{\overline{M}^T}\right)
\end{equation} 
 corresponds to a positive hermitian matrix of size $m\times m$ and the coefficients for $\sqrt{-1}dz_i\wedge d\bar z_j$ is given by 
$$\Gamma_{ij}= \left(\sum_{k=1}^m \lambda_k^{-1} \right) \delta_{ij} \mu_i    - \sum_{k, l=1, l \neq k}^m( \lambda_k \lambda_l)^{-1}\eta_{i\bar l} \eta_{l \bar j}.$$
Then the matrix $\Gamma$ of size $m\times m$ is bounded below by the following calculations.
\begin{eqnarray*}
\Gamma &= &\left[  \left(\sum_{k=1}^m \lambda_k^{-1} \right) \delta_{ij} \mu_i    - \sum_{k, l=1}^m( \lambda_k \lambda_l)^{-1}\eta_{i\bar l} \eta_{l \bar j} \right]_{i, j}+ \left[ \sum_{l=1}^m(  \lambda_l^2)^{-1}\eta_{i\bar l} \eta_{l \bar j} \right]_{i, j}\\  
&\geq & \left[ \left(\sum_{k=1}^m \lambda_k^{-1} \right) \delta_{ij} \mu_i    - \sum_{k, l=1}^m( \lambda_k \lambda_l)^{-1}\eta_{i\bar l} \eta_{l \bar j} \right]_{ij}\\  
&= &\left(\sum_{k=1}^m \lambda_k^{-1} \right) \left( \delta_{ij} \mu_i    - \sum_{l=1}^m  \lambda_l^{-1}\eta_{i\bar l} \eta_{l \bar j} \right).
 \end{eqnarray*}
We can apply Lemma \ref{s3con2} and 
\begin{eqnarray*}
&&\cP_{  \pi_1^* \hat \chi } ( \left( tr_{\omega_V}(\chi_V)\right)^{-1} \Gamma ) \\
&\leq&  \cP_{\chi_\cZ} ( \omega_{\cZ, s}) - tr_{\omega_V}(\chi_V) \\
&<&   (m+n) - tr_{\omega_V}(\chi_V)  ,
\end{eqnarray*}
where $\cP_{ \pi_1^* \hat \chi } ( \Gamma )$ is taken on the horizontal $m$-dimensional space $\hat Z$ for components in $\sqrt{-1} dz_i\wedge d\bar z_j$. 
Since $$  \omega_{s} = \cV^{-1} ( \pi_1)_*\left( \Gamma \wedge \omega_V^m\right), ~  \cV^{-1}    \int_{\pi_1^{-1}(z)}  tr_{\omega_V}(\chi_V)  \omega_V^m=1, $$
by the convexity of $\cP $ operator for fixed $z\in \hat Z$, we have
\begin{eqnarray*}
&&\cP_{\hat\chi} (\omega_{s}) \\
&= &  \cP_{\hat\chi} \left(   \cV^{-1} \int_{\pi_1^{-1}(z)} \Gamma \wedge \omega_V^m\right)  \\
&\leq & \cV^{-1}    \int_{\pi_1^{-1}(z)}  tr_{\omega_V}(\chi_V)\cP_{\hat\chi} \left( \ \left( tr_{\omega_V}(\chi_V)\right)^{-1}\Gamma\right) \wedge \omega_V^m  \\
&<& \cV^{-1}\left(  (m+n) \int_{\pi_1^{-1}(z)}   tr_{\omega_V}(\chi_V) \omega_V^m -  \int_{\pi_1^{-1}(z)}   \left( tr_{\omega_V}(\chi_V) \right)^2 \omega_V^m \right)  \\
&\leq & (m+n) - \cV^{-1} \frac{ \left( \int_{\pi_1^{-1}(z)}   tr_{\omega_V}(\chi_V) \omega_V^m\right)^2}{   \int_{\pi_1^{-1}(z)}    \omega_V^m }\\
&=& n.
\end{eqnarray*} 

\end{proof}

Let $\Delta_\eta$ be the $\eta$-neighborhood of $\Delta$ in $\cZ$. 
We define the following $(1,1)$-forms on $\hat Z$
\begin{equation}
\omega_{s, \eta} (z)=\cV^{-1} \int_{\pi_1^{-1}(z)\cap \Delta_\eta} (\omega_{\cZ, s})^m \wedge \chi_V,
\end{equation}
\begin{equation}
\omega'_{s, \eta} (z)= \cV^{-1} \int_{\pi_1^{-1}(z)\cap \Delta_\eta}   \chi_\cZ \wedge (\omega_V)^{m-1} \wedge \chi_V
\end{equation}
and
\begin{eqnarray}
\Omega_{s, \eta} (z) &=&\omega_{s} - \omega_{s, \eta} +   \omega'_{s, \eta} \nonumber
\\
&=& \cV^{-1} \left( \int_{\pi_1^{-1}(z)\setminus \Delta_\eta} (\omega_{\cZ, s})^m \wedge \chi_V
+      \int_{\pi_1^{-1}(z)\cap\Delta_\eta}  \chi_\cZ \wedge (\omega_V)^{m-1} \wedge \chi_V\right) . 
\end{eqnarray}

\begin{lemma} \label{s5l56} Under the same assumption of Lemma \ref{s5l55}, we have
\begin{equation}
\cP_{\hat\chi} (\Omega_{s, \eta}) \leq  n + m\cV^{-1}  \int_{\pi_1^{-1}(z)\cap \Delta_\eta} tr_{\omega_V}(\chi_V) (\omega_V)^m ,  
\end{equation}
\end{lemma}

\begin{proof} By defnition of $\Omega_{s, \eta}$ and $\Gamma$,  we have
$$ 
\Omega_{s,\eta} =  \cV^{-1} \int_{\pi_1^{-1}(z)\setminus \Delta\eta} \Gamma \wedge (\omega_V)^m  +   \cV^{-1}  \int_{ \pi_1^{-1}(z)\cap \Delta\eta} \chi_H \wedge (\omega_V)^{m-1}\wedge \chi_V.
$$

By Jensen's inequality and the calculations in the proof of Lemma \ref{s5l55},  we have 
\begin{eqnarray*}
&&\cP_{\hat\chi}(\Omega_{s, \eta}) \\
&\leq&   \cV^{-1} \int_{\pi_1^{-1}(z)\setminus \Delta\eta} \cP_{\chi_H} \left( tr_{\omega_V}(\chi_V)^{-1}\Gamma \right) tr_{\omega_V}(\chi_V)(\omega_V)^m \\
&&  + m \cV^{-1}  \int_{\pi_1^{-1}(z)\cap \Delta_\eta}\cP_{\chi_H}( m^{-1}\chi_H) \wedge (\omega_V)^{m-1}\wedge \chi_V\\
&\leq & \cV^{-1} \left( (m+n) \int_{\pi_1^{-1}(z)\setminus \Delta_\eta} tr_{\omega_V}(\chi_V)(\omega_V)^m -  \int_{\pi_1^{-1}(z)\setminus \Delta_\eta} \left( tr_{\omega_V}(\chi_V) \right)^2 (\omega_V)^m \right)\\
&& + m \cV^{-1}  \int_{\pi_1^{-1}(z)\cap \Delta_\eta} tr_{\omega_V}(\chi_V) (\omega_V)^m\\
&\leq&   n + \cV^{-1}  \int_{\pi_1^{-1}(z)\cap \Delta_\eta} \left( tr_{\omega_V}(\chi_V) \right)^2 (\omega_V)^m  + (m- m-n )\cV^{-1}  \int_{\pi_1^{-1}(z)\cap \Delta_\eta} tr_{\omega_V}(\chi_V) (\omega_V)^m\\
& \leq & n + m\cV^{-1}  \int_{\pi_1^{-1}(z)\cap \Delta_\eta} tr_{\omega_V}(\chi_V)(\omega_V)^m.
\end{eqnarray*}
The last inequality follows from the fact that $tr_{\omega_V} (\chi_V) \leq tr_{\omega_{\cZ, s}}(\chi_\cZ) \leq m+n$.

\end{proof}

Recall $\theta$ is a K\"ahler metric on $\hat Z$. Then we have the following lemma.
\begin{lemma} \label{s5l57} For any $\varepsilon>0$, $t\in (0,1)$ and $\epsilon\in (0, \epsilon_Z)$, there exist sufficiently small $\eta=\eta(\varepsilon, t, \epsilon)>0$ and $s_0=s_0(\varepsilon, t, \epsilon, \eta)\in (0,1)$ such that for all $s\in (0, s_0)$, 
\begin{equation}
 \int_{ \Delta_\eta}  (\omega_V)^{m-1} \wedge \chi_V \wedge \pi_1^*(\theta^m) < \varepsilon.
\end{equation}

\end{lemma}

\begin{proof} We prove by contradiction. Suppose there exist $\varepsilon'>0$, $t\in (0,1)$, $\epsilon\in (0, \epsilon_Z)$ and  sequences $\eta_j \rightarrow 0$, $s_j \rightarrow 0$ such that for each $j$ and $\omega_V=\omega_V(t, \epsilon, s_j)$, we have
 $$\int_{ \Delta_{\eta_j}}  (\omega_V)^{m-1} \wedge \chi_V \wedge \pi_1^*(\theta^m) \geq 
 \varepsilon'.$$
We note that  
$$ \int_{ \Delta_\eta}  (\omega_V)^{m-1} \wedge \chi_V \wedge \pi_1^*(\theta^m) =  \int_{ \Delta_\eta}  (\omega_{\cZ, s})^{m-1} \wedge \chi_V \wedge \pi_1^*(\theta^m).
$$
Suppose  $(\omega_{\cZ, s_j})^{m-1}$ converges to a  closed positive  $(m-1, m-1)$-current $\cT$ on $\cZ$ as $s_j\rightarrow 0$, after passing to a sequence. By Siu's decomposition theorem, we have
$$\cT = \sum_{i=1}^\infty a_i [E_i] + \cR $$
where $a_i>0$, $E_j$ is an analytic subvariety of $\cZ$ with 
$$\dim E_i \leq m+1$$
and $\cR$ is the residue current with vanishing Lelong number everywhere. After taking a subsequence, for any fixed $\eta>0$
\begin{eqnarray*}
&&\lim_{s_j\rightarrow 0} \int_{\Delta_\eta} (\omega_V)^{m-1}\wedge\chi_V\wedge \pi_1^*(\theta^m)\\
&=& \sum_{i=1, \dim E_i = m+1}^\infty  a_i \int_{\Delta_\eta\cap E_i} \chi_V\wedge \pi_1^*(\theta^m) + \int_{\Delta_\eta} \cR\wedge \chi_V\wedge \pi_1^*(\theta^m).
\end{eqnarray*}

Since $\int_\cZ (\omega_V)^{m-1} \wedge \chi_V \wedge \theta^m$ is uniformly bounded above, for any $\epsilon>0$, there exists $J>0$ such that 
$$\sum_{j>J} a_j \int_{E_j} \chi_V \wedge \pi_1^* (\theta^m) < \varepsilon. $$
Since $\cR$ has vanishing Lelong number, by definition and by choosing $\eta>0$ sufficiently small, 
$$\int_{\Delta_\eta} \cR \wedge \chi_V\wedge \pi_1^*(\theta^m) < \varepsilon. $$
Since $\chi_V\wedge \pi_1^*(\theta^m) $ is a fixed smooth form on $\cZ$, by choosing $\eta$ sufficiently small, we have
$$a_i \int_{\Delta_\eta\cap E_=i }\chi_V\wedge \pi_1^*(\theta^m) < \varepsilon $$
for all $i$ with $\dim E_i = m+1$.

Combining the above estimates, we have
$$\lim_{s_j\rightarrow 0} \int_{\Delta_\eta} (\omega_V)^{m-1}\wedge\chi_V\wedge \theta^m= \sum_i a_i \int_{\Delta_\eta\cap E_i}  \chi_V\wedge \theta^m + \int_{\Delta_\eta} \cR\wedge \chi_V \wedge \theta^m< 3\varepsilon. $$
Then for sufficiently small $s_j$, we have
$$ \int_{\Delta_\eta} (\omega_V)^{m-1}\wedge\chi_V\wedge \theta^m < 4\varepsilon. $$
This leads to contradiction by choosing $4\varepsilon < \varepsilon'$ and the lemma is proved.
\end{proof}

\begin{lemma}  There exists $r_0>0$ such that for any $\varepsilon>0$, $t\in (0,1)$ and $\epsilon\in (0, \epsilon_Z)$,  there exist sufficiently small $\eta=\eta(r_0, \varepsilon, t, \epsilon), s_0= s_0(r_0, \varepsilon, t, \epsilon, \eta)>0$ such that for any $s\in (0, s_0)$, $r\in (0, r_0)$, $j\in \cJ$, we have in $B_j$
\begin{equation}
\cP_{\hat\chi_j} \left( \Omega_{s, \eta}^{(r)} \right) \leq n + \varepsilon,
\end{equation}
where $\hat\chi_j$ is defined in (\ref{s5epchoi}) and $\Omega_{s, \eta}^{(r)}$ is the regularization in $2B_j$. 

\end{lemma}

\begin{proof} 

Since $\hat\chi_j$ has constant coefficients, for any $z\in B_j$, by convexity of $\cP_{\hat\chi_j}(\cdot)$ we have 
\begin{eqnarray*}
&& \cP_{\hat\chi_j} \left( \Omega_{s, \eta}^{(r)} \right) (z) \\
&=& \cP_{\hat\chi_j} \left( \int_{y\in B_r(0)} r^{-2m} \rho(r^{-1}|y|) \Omega_{s, \eta}(z+y) d\textnormal{Vol}_{\mathbb{C}^m} (y) \right)\\
&\leq& \int_{B_r(0)}   r^{-2m} \rho(r^{-1}|y|) \cP_{\hat\chi_j} \left(\Omega_{s, \eta}(z+y) \right) d\textnormal{Vol}_{\mathbb{C}^m} (y)  \\
&\leq& \int_{B_r(0)}   r^{-2m} \rho(r^{-1}|y|)\left( n + m \cV ^{-1}  \int_{\pi_1^{-1}(z+y)\cap \Delta_\eta} tr_{\omega_V}(\chi_V) (\omega_V)^m\right)d\textnormal{Vol}_{\mathbb{C}^m} (y)  \\
&\leq&n + m2^m \cV^{-1}  \int_{  \Delta_\eta} (\omega_V)^{m-1} \wedge \chi_V\wedge \theta^m\\
&\leq& n + \varepsilon
\end{eqnarray*}
by choosing sufficiently small $\eta, s>0$.
Here the second inequality follows from Lemma \ref{s5l56} and the last inequality follows from Lemma \ref{s5l57}. 
\end{proof}

\begin{lemma} \label{s5l59} There exists $r_0>0$ such that for any $t\in (0, 1)$, $\epsilon\in (0,\epsilon_Z)$ and $\eta\in (0, 1)$, there exists $s_0\in (0,1)$ so that for any $s\in (0, s_0)$ and $r\in (0, r_0)$, we have
\begin{equation}
\left( \omega_{s, \eta}+ 100\delta_0 \ddbar \phi\right) ^{(r)} > 10 \delta_0  \theta, 
\end{equation}
where $\phi$ is defined in (\ref{s5phi}).

\end{lemma}

\begin{proof} For any $z\in \hat Z$, for any sequence $s_j \rightarrow 0$, after passing to a subsequence, we have $(\omega_{\cZ, s_j})^ m $ weakly converges to a closed positive current $\Theta$ and by Lemma \ref{s5mass}, 
\begin{eqnarray*}
&& \lim_{s_j\rightarrow 0}  \omega_{s_j, \eta} ^{(r)}  (z) \\
&=&    \lim_{s_j\rightarrow 0}\int_{y\in B_r(0)}   r^{-2m} \rho(r^{-1}|y|) \omega_{s_j, \eta} (z+y) d\textnormal{Vol}_{\mathbb{C}^m} (y)\\
&=&  \cV^{-1} \lim_{s_j\rightarrow 0}\int_{y\in B_r(0)}   r^{-2m} \rho(r^{-1}|y|)  \left( \int_{ w\in \pi_1^{-1}(z+y) \cap \Delta_\eta} (\omega_{\cZ, s_j})^ m(z+y, w)\wedge \chi_V(w) \right) d\textnormal{Vol}_{\mathbb{C}^m} (y)\\
&= &  \cV^{-1}\int_{(z', w)\in (B_r(z) \times \cZ )\cap \Delta_\eta}   \left( r^{-2m} \rho(r^{-1}|z'- z|)  \Theta(z', w) \wedge \chi_V(w)\right) d\textnormal{Vol}_{\mathbb{C}^m} (z')\\
&\geq & 100\delta_0    \int_{(z', w)\in (B_r(z) \times \cZ )\cap \Delta }  \left(r^{-2m} \rho(r^{-1}|z'- z|)   \chi_V(w)\right) d\textnormal{Vol}_{\mathbb{C}^m} (z'). \\
&=& 100\delta_0    \int_{ z'\in B_r(z)  }  r^{-2m} \rho(r^{-1}|z'- z|)  \hat \chi(z') d\textnormal{Vol}_{\mathbb{C}^m} (z')
\end{eqnarray*}
 because $ r^{-2m} \rho(r^{-1}|z'- z|)  d\textnormal{Vol}_{\mathbb{C}^m} (z')$ is independent of $s_j$.  By definition of $\phi$, we have 
\begin{eqnarray*}
&&  \lim_{s_j\rightarrow 0}  \left( \omega_{s_j, \eta} + 100\delta_0 \ddbar  \phi\right)^{(r)}  (z) \\
&\geq& 100 \int_{z'\in B_r(z)  }  \left( r^{-2m} \rho(r^{-1}|z'- z|)  ( \delta_0 \hat\chi+ \delta_0\ddbar \phi)(z') \right) d\textnormal{Vol}_{\mathbb{C}^m} (z')\\
&\geq& 99 \delta_0 ~\theta
\end{eqnarray*}
by choosing $r\in (0, r_0)$ for some uniform  small $r_0$.

We claim that for fixed $t$, $\epsilon$, any $r\in (0,r_0)$ for some $r_0>0$ independent of $t, \epsilon$ and for sufficiently small $s> 0$ (possibly dependent on $t, \epsilon$),  
\begin{equation}\label{s5molli}
\left( \omega_{s, \eta}+ 100\delta_0 \ddbar \phi \right) ^{(r)} > 10\delta_0 \theta.
\end{equation}
 Suppose (\ref{s5molli}) fails at $s_j$ and $z_j$ for a sequence $s_j\rightarrow 0$ and $z_j \rightarrow z$. 
Then by passing to a subsequence and by similar calculation for $\lim_{s_j\rightarrow 0} \omega_{s, \eta}^{(r)}$, we have
\begin{eqnarray*}
&& \lim_{j \rightarrow \infty}\left( \omega_{s_j, \eta} + 100\delta_0 \ddbar  \phi\right) ^{(r)} (z_j) \\
&=&  \cV^{-1} \lim_{j \rightarrow \infty}\int_{y\in B_r(0)}   r^{-2m} \rho\left(\frac{|y|}{r}\right) \left( \int_{ w\in \pi_1^{-1}(z_j+y) \cap \Delta_\eta} (\omega_{\cZ, s_j})^ m(z_j+y, w)\wedge \chi_V(w)\right) d\textnormal{Vol}_{\mathbb{C}^m} (y)\\
&&+ 20\delta_0 \lim_{j\rightarrow \infty} (\hat\chi + \ddbar \phi)^{(r)}(z_j) -20\delta_0 \lim_{j\rightarrow \infty} \hat\chi ^{(r)}(z_j)\\
&\geq & \cV^{-1} \lim_{j \rightarrow \infty}\int_{y\in B_{r-|z-z_j|}(0)}   r^{-2m} \rho\left(\frac{|y+z_j-z|}{r}\right)  \left(\int_{ w\in \pi_1^{-1}(z+y) \cap \Delta_\eta} (\omega_{\cZ, s_j})^ m(z+y, w)\wedge \chi_V(w)\right)d\textnormal{Vol}_{\mathbb{C}^m} (y)\\
&&+  20\delta_0 \theta^{(r)}(z) - 20\delta_0 (\hat\chi)^{(r)}(z) \\
&\geq&  \cV^{-1} \lim_{j\rightarrow\infty} \int_{y\in B_r(0)}   r^{-2m} \rho\left(\frac{|y|}{r}\right) \left( \int_{ w\in \pi_1^{-1}(z+y) \cap \Delta_\eta} (\omega_{\cZ, s_j})^ m(z+y, w)\wedge \chi_V(w)\right) d\textnormal{Vol}_{\mathbb{C}^m} (y) \\
&& - 2^m \cV^{-1} \left( \lim_{j \rightarrow \infty} \sup_{y\in B_r(0)}  r^{-2m} \left| \rho(r^{-1}|y|)- \rho(r^{-1}|y+z_j-z|) \right|  \right)  \left(\int_\cZ (\omega_{\cZ, s_j})^m \wedge \chi_V\wedge \theta^{m-1}\right) \theta(z) \\
&& +  20\delta_0 \theta^{(r)}(z) -20 \delta_0 (\hat\chi)^{(r)}(z)  \\
&\geq &  15 \delta_0 \theta (z) 
\end{eqnarray*}
for sufficiently large $j>0$  and $r\in(0, r_0)$. This leads to  contradiction and the lemma is proved. 
\end{proof}

\begin{lemma} \label{s5l510} There exists $r_0>0$ such that for any $\varepsilon>0$,  $t\in (0,1)$ and $\epsilon\in (0, \epsilon_Z)$, there exist $s_0>0$ and $\eta_0>0$  such that for any $s\in (0, s_0)$, $\eta\in (0, \eta_0)$ and $r\in (0, r_0)$, 
\begin{equation}
(\omega'_{s, \eta})^{(r)} < \varepsilon\theta. 
\end{equation}

\end{lemma}

\begin{proof} The proof can be easily obtained by combining Lemma \ref{s5l57} and the argument in Lemma \ref{s5l59}. 
\end{proof}

For any small $\varepsilon>0$,  we will  choose  sufficiently small $t_0=t_0(\varepsilon) \in (0,1)$ and $\epsilon_0=\epsilon(\varepsilon) \in (0, 
\epsilon_Z)$, such that in each $2B_j$ from the covering, we have 
\begin{equation}\label{choepp}
\hat\chi_j \leq \hat \chi \leq \hat\chi_j + 2\hat\epsilon g_j,   ~~\hat\chi < 2\theta, ~~2^{-1}\theta\leq g_j \leq 2\theta
\end{equation}
for some fixed $\hat\epsilon>0$ satisfying
\begin{equation}\label{choep}
\hat\epsilon=\hat\epsilon(\varepsilon)< m^{-1} \delta_0\varepsilon, 
\end{equation}
where $\delta_0$ is given in Lemma \ref{s5mass}. We remark that $\theta_0$  does not depend on the choice of $\hat \epsilon \in (0, 1)$. From now on, we fix $\theta_0$ and for fixed $\varepsilon>0$ we will choose the covering $\{ B_j\}_{j\in \cJ}$ and $\hat\epsilon$ satisfying and (\ref{choepp}) and (\ref{choep}).  

\begin{proposition} \label{s5prop51} There exists $r_0>0$ such that for any $\varepsilon>0$, there exist $t_0=t_0(\varepsilon)>0$ and  $\epsilon_0=\epsilon_0(\varepsilon)>0$, such that for any $t\in (0, t_0)$ and $\epsilon \in (0,\epsilon_0)$, there exist $s_0>0$  so that for any $s\in (0, s_0)$, we have on $\hat Z$, 
\begin{equation}\label{est545}
\cP_{\hat\chi + \hat\epsilon \theta }\left( \left(\omega_{s}  - \lambda^{-1}\delta_0 \hat\omega + 100\delta_0 \ddbar \phi ) \right)^{(r)}\right)<   n+2 \varepsilon . 
\end{equation}

\end{proposition}

\begin{proof}  

By Lemma \ref{s5l59} and Lemma \ref{s5l510}, for any $r\in (0, r_0)$, $t\in (0, t_0)$ and $\epsilon\in (0, \epsilon_0)$, there exists sufficiently small $\eta>0$ such that for any sufficiently small $s>0$, we have  
\begin{eqnarray*}
&&\left( \omega_{s} - \lambda^{-1} \delta_0 \hat \omega + 100\delta_0 \ddbar  \phi\right)^{(r)} \\
&=& \left( \Omega_{s, \eta} + (\omega_{s,\eta} + 100 \delta_0 \ddbar \phi) -  \lambda^{-1}\delta_0 \hat \omega-  \omega'_{s, \eta}    \right)^{(r)} \\
&\geq &   \Omega_{s, \eta}^{(r)} +  10\delta_0 \theta - 2\delta_0 \theta  - \delta_0 \theta      \\
&>& 0.
\end{eqnarray*}
By the choice of $\hat\epsilon$ as in (\ref{choep}), we have 
\begin{eqnarray*}
&&\cP_{\hat\chi + \hat \epsilon g_j}\left( \left(\omega_{s}  - \lambda^{-1} \delta_0 \hat\omega + 100\delta_0 \ddbar  \phi ) \right)^{(r)}\right) \\
&\leq& \cP_{\hat\chi_j +  \hat \epsilon \theta} \left(  \Omega_{s, \eta}^{(r)}  +  5\delta_0 \theta \right)\\
&\leq&   \cP_{\hat\chi_j}\left(\Omega_{s, \eta}^{(r)} \right) +  \cP_{  \hat\epsilon\theta} (5\delta_0 \theta)  \\
&\leq&  n+\varepsilon+ \frac{m\hat\epsilon}{\delta_0}  \\
&\leq& n +2\varepsilon. 
\end{eqnarray*}
\end{proof}

Estimate (\ref{est545}) in Proposition \ref{s5prop51} immediately implies that 
\begin{equation}
\cP_{\chi }\left( \left(\omega_{s}  - 2K^{-1}\delta_0 \hat\omega + 100\delta_0 \ddbar \phi ) \right)^{(r)}\right)<   n+2\varepsilon . 
\end{equation}
 by the choice of $\chi_j$.  Now we can now complete the proof of Theorem \ref{masscon}.  

\begin{proof}[Proof of Theorem \ref{masscon}]   We fix a sufficiently small $\varepsilon > 0$, by Proposition \ref{s5prop51}, there exist $r_0>0$ such that for fixed $t\in (0, t_0(\varepsilon))$ and $\epsilon\in (0, \epsilon_0(\epsilon))$,  there exist a sequence $s_i \rightarrow 0$ such that for any $r\in (0, r_0)$, 
$$\cP_{\hat\chi+ \theta} \left( (\omega_{s_i} - \lambda^{-1}\delta_0 \hat\omega  + 100\delta_0 \ddbar\phi)^{(r)} \right) \leq n+ 2\varepsilon.$$
After possibly passing to a subsequence, we let
$$\tilde\omega=\tilde\omega(t, \epsilon) = \lim_{s_i \rightarrow 0} (\omega_{s_i} - \lambda^{-1}\delta_0 \hat\omega  + 100\delta_0 \ddbar\phi).$$
Obviously $\tilde\omega(t, \epsilon)\in (1- \lambda^{-1}\delta_0)\hat\alpha$ and by Lemma \ref{s5l59},
\begin{eqnarray*}
\tilde\omega(t, \epsilon) &\geq&  100\delta_0 \hat\chi - \lambda^{-1} \delta_0 \hat\omega+ 100 \delta_0 \ddbar\phi\\
&\geq& 10\delta_0 \theta - 4\delta_0\theta\\
&\geq & 5\delta_0 \theta.
\end{eqnarray*}
On each $B_j$, after passing to a subsequence, there exist  plurisubharmonic functions $\psi_{s_i}$ and $\psi$ such that %
$$\omega_{s_i} - \lambda^{-1}\delta_0 \hat\omega  + 100\delta_0 \ddbar  \phi= \ddbar \psi_{s_i},  ~\ddbar \tilde\psi = \tilde \omega$$
and
$$~\lim_{s_i\rightarrow 0} \psi_{s_i} = \tilde\psi$$
in $L^1$ locally. Then for any $r\in (0, r_0)$, $\psi_{s_i}^{(r)}$ converges to $\tilde\psi^{(r)}$ uniformly in any compact subset of $B_j$. Since in $B_j$, 
$$(n+2\varepsilon) \left(\ddbar\psi_{s_i}^{(r)}\right)^m \geq  m \left(\ddbar \psi_{s_i}^{(r)}\right)^{m-1}\wedge \hat\chi\geq 0, $$
by letting $i\rightarrow \infty$, we have for any $t\in (0, t_0(\varepsilon)$, $\epsilon\in (0, \epsilon_0(\varepsilon))$ and $r\in (0, r_0)$, 
$$(n +2\varepsilon) \left( \tilde\omega^{(r)} \right)^m = n\left(\ddbar\tilde\psi^{(r)}\right) \geq  m\left(\ddbar\tilde\psi^{(r)}\right)^{m-1}\wedge \hat\chi= m \left( \tilde\omega^{(r)}\right)^{m-1}\wedge \hat \chi.$$
Now by taking a sequence $t_k, \epsilon_k \rightarrow 0$, we can assume
$\tilde\omega(t_k, \epsilon_k)$ converges to a closed positive current $\tilde\omega_0 \in (1-\lambda^{-1} \delta_0)\alpha$. In particular, $\tilde\omega_0$ has strictly positive Lelong number along $\cS_{\hat Z}$ since $\phi$ has strictly  positive Lelong number along $\cS_{\hat Z}$.

By the same argument above, we have for any $r\in (0, r_0)$, 
$$(n+2\varepsilon) \left( \tilde\omega_0^{(r)} \right)^m \geq m \left( \tilde\omega_0^{(r)}\right)^{m-1}\wedge \chi.$$
We let $\Omega = (1+ \lambda^{-1}  \delta_0)\tilde\omega_0$. Then
$$\Omega\in (1-\lambda^{-2}\delta_0^2) \alpha, ~ \cP_{\chi}(\Omega) \leq \frac{n+2\varepsilon}{(1+\lambda^{-1}\delta_0)} . $$
This completes the proof of Theorem \ref{masscon} by choosing $\delta = \lambda^{-2}\delta_0^2$ and $\varepsilon<< \lambda^{-1} \delta_0$ sufficiently small.
\end{proof}

\bigskip


\section{Gluing argument} 

In this section, we will follow the idea of Chen \cite{ChG} for gluing the local potentials by induction assumption and the trick of Blocki-Kolodziej \cite{BK}. The main difference of our argument from \cite{ChG} is that we do not assume the pluriclosed set of the modified plurisubharmonic function is a smooth subvariety and we also have to extend such a function. 

We keep the same notations as in \S 3, \S 4 and \S 5. The following is the main result of this section.

\begin{theorem} \label{s6main} Under the same assumptions of Theorem \ref{s3main}, we let $Z$ be an $m$-dimensional analytic subvariety of $X$ and let $\cS_Z$ be the set of singular points  of $Z$. Then there exists   $\varphi_Z\in C^\infty(Z\setminus   \cS_Z)$ such that
\begin{enumerate}
\item $\omega_Z = \omega_0+\ddbar\varphi_Z$
is a K\"ahler form on $Z\setminus   \cS_Z$ satisfying
$$n(\omega_Z)^m- m (\omega_Z)^{m-1}\wedge \chi>0,$$
\item  $\varphi_Z$ tends to $-\infty$ uniformly at $\cS_Z$.

\end{enumerate}

\end{theorem}
In fact, $\varphi_Z$ in Theorem \ref{s6main} has positive Lelong number along $\cS_Z$. We first recall the following well-known definition (c.f. \cite{BK}) for Lelong number.
\begin{definition}\label{lelo} Let $\varphi$ be a plurisubharmonic function on a domain $\mathcal{U} \in \mathbb{C}^m$. We define for any $p \in \mathcal{U}$,
\begin{equation}
\nu_{\varphi}(p, r) = \frac{ \varphi_R(p)- \varphi_{r}(p)}{\log R -\log r},  
\end{equation}
where $0<r< R$, $B_R(p)\subset\subset \mathcal{U}$ and $\varphi_r(p)$ is defined by
\begin{equation}
\varphi_r(p)= \max_{B_r(p)} \varphi.
\end{equation}

\end{definition}

As $r \rightarrow 0$, $\nu_{\varphi}(p, r)$ converges decreasingly to the Lelong number of $\varphi$ at $x$
$$\lim_{r\rightarrow 0 } \nu_{\varphi}(p, r) = \nu_{\varphi}(p).$$

We let $\hat Z$ be a fixed component of the strict transform of $Z$ by $\Phi$ as in \S 4 and \S5. We keep the same notations for $\chi, \omega_0$ and $\theta$ as before.  For any small $\delta>0$, we will fix a covering of $\hat Z$ by finitely many Euclidean balls $\{ B_{i,4R}=B_{4R}(p_i)\}_{i\in \cI, p_i \in \hat Z}$ such that in each $B_{i, 4R}$, we have 
\begin{equation}
\theta = \ddbar \phi_{i,\theta}, ~  | \phi_{i, \theta} - r^2| \leq \tau R^2, ~r=|z|^2, 
\end{equation}
\begin{equation}
\omega_0= \ddbar \phi_{i, \omega_0}, ~   |\nabla \phi_{i, \omega_0}| \leq   K R, \phi_{i, \omega_0}(p_i)=0, 
\end{equation}
\begin{equation}
\chi = \ddbar \phi_{i,\chi}, ~   |\nabla \phi_{i, \chi}| \leq   K R, \phi_{i, \chi}(p_i)=0
\end{equation}
for some fixed sufficiently large  $K>0$  and sufficiently small $\tau>0$ independent of $R$,  where $\{ p_i\}_{i\in \cI}$ is a set of finitely many points on $\hat Z$ and $z$ is the local holomorphic coordinates in $B_{i, 4R}$. Without loss of generality, we can always assume $\{ B_{i,4R}\}_i$ are  finer coverings of $\{B_j\}_{j\in\cJ}$ in \S 5. Furthermore, we can require $4R<r_0$, where $r_0>0$ is given in Theorem \ref{masscon}. 

 We  choose $\Omega$ from Theorem \ref{masscon} and we let 
\begin{equation}
\Omega = (1-\delta) \omega_0 + \ddbar \varphi,
\end{equation}
for some $\delta>0$ and $\varphi\in \textnormal{PSH}(\hat Z, (1- \delta)\omega_0)$.  Then for some $\varepsilon>0$, 
$$\cP_{\chi}(\Omega^{(r)}) < n - \varepsilon$$
for all $0<r < R< r_0$.

We also define the following local plurisubharmonic functions at any $p\in B_{i, 3R}$ for any $r< R$,

\begin{equation}
\varphi_i (p)= \varphi (p)+ (1- \delta) \phi_{i,\omega_0}(p)
\end{equation}
and
\begin{equation}
  \varphi_{i, r} (p)= \sup_{B_{i,r}(p)} ( \varphi + (1-\delta) \phi_{i, \omega_0} ),
 \end{equation}
where $B_{i, r}(p)$ is the Euclidean ball in $B_{i, 4R}$ since $p\in B_{i, 3R}$ and $r<R$. 
Then by the choice of $\Omega$ above from Theorem \ref{masscon}, we  immediately have the following lemma.
\begin{lemma} \label{s6l61} There exists $\varepsilon>0$ such that in each $B_{i,3R}$, we have 
$$ \cP_{\chi } \left(  \ddbar \varphi_i^{(r)} \right) < n-\varepsilon$$
for any $0<r<R$, 
where $\varphi_i^{(r)}$ is the local regularization of $\varphi_i$ in $B_{i, 4R}$. 

\end{lemma}
We note that $\nu_{\varphi}(p) = \nu_{\varphi_i}(p)$ for any $p\in B_{i, 3R}$ since $\phi_{i, \omega_0}$ is smooth.    The following lemma from \cite{ChG} (Lemma 4.2) gives some basic properties of $\nu_{i, \varphi_i}(p, r)$ in each $B_{i, 3R}$.  It can be proved using log convexity and Poisson kernel as in \cite{BK}. 
 
\begin{lemma} \label{s6l62}For $r < R$ and $p\in B_{i, 3R}$,  the following estimates hold. 
\medskip

\begin{enumerate}
\item 
$0 \leq  \varphi_{i, r} (p)-  \varphi_{i, \frac{r}{2}}(p) \leq (\log 2 )   \nu_{i, \varphi_i}(p, r). $

\medskip

\item   $ 0 \leq  \varphi_{i, r} (p)- \varphi_i^{(r)}(p) \leq \eta \nu_{i, \varphi_i} (p, r), $

\medskip
\end{enumerate}
where $\nu_{i, \varphi_i}(p, r)$ is defined for $\varphi$ in $B_{i, 3R}$ and $\eta>0$ is defined by
$$\eta =2^{2m} - \Vol(\partial B_1(0)) \int_0^1 t^{2m-1} (\log t) \rho(t) dt. $$ 

\end{lemma}

We let $\cS_{\tilde\epsilon}$ be the analytic subvariety of $\hat Z$ defined by
\begin{equation}
\cS_{\tilde\epsilon} = \{ p \in \hat Z~|~ \nu_{\varphi}(p) \geq  \tilde\epsilon\} 
\end{equation}
for a fixed $\tilde\epsilon>0$ satisfying 
\begin{equation}\label{lelongcho}
\tilde \epsilon < \min\left(  \frac{\delta^3 R^2}{6+4\eta},~  \inf_{p\in \Phi^{-1}(\cS_{\hat Z})}\frac{\nu_\varphi(p)}{4}  \right)  .
\end{equation}
By Siu's decomposition theorem, $\cS_{\tilde{\epsilon}}$ is an analytic subvariety of $\hat Z$ containing $\cS_{\hat Z}$. 

By induction of dimensions, there exist an open neighborhood $U$ of $\Phi(\cS_{\tilde\epsilon})$ in $\cM$ and a smooth K\"ahler metric $\omega_U \in \alpha|_U$ such that
$$n(\omega_U)^{n-1} - (n-1)(\omega_U)^{n-2}\wedge\chi>0$$
in $U$.
We let $\hat U = \Phi^{-1} (U)$ and 
$$\omega_{\hat U}=\Phi^* \omega_{\hat U}= \omega_0 + \ddbar \varphi_{\hat U} $$
for some uniformly bounded and smooth $\varphi_{\hat U}$ (after possibly shrinking $U$).
Then immediately, we have 
$$n(\omega_{\hat U})^{n-1} - (n-1)(\omega_{\hat U})^{n-2}\wedge\chi \geq 0 $$
 in $\hat U$ and the strict inequality holds  in $\hat Z \setminus \cS_{\hat Z}$.

We define $\tilde\varphi_{i, r}$ on $B_{i, 3R}$ by 
\begin{equation}
\tilde \varphi_{i, r} (p) = \varphi_i^{(r)}(p)  - (1- \delta) \phi_{i, \omega_0}(p) -  \delta^3  \phi_{i, \theta}(p) +  \delta^2 \phi %
\end{equation}
Since we can choose a fixed sufficiently small $\delta$, we will assume that 
$$\delta^2 \theta <  \omega_0 +\delta \ddbar \phi.$$ 
Immediately we have 

\begin{equation}
\omega_0 +\ddbar  \tilde \varphi_{i, r} \geq \Omega^{(r)} -\delta^3 \theta + \delta(\omega_0 + \delta\ddbar \phi) > \Omega^{(r)}
\end{equation}
and
\begin{equation}
\cP_{\chi} ( \omega_0 + \ddbar \tilde \varphi_{i, r} ) < n - \varepsilon
\end{equation}
in each $B_{i, 4R}$ if we choose sufficiently small $\delta>0$.

Let $r_0$ be defined in Theorem \ref{masscon}. The following lemma corresponds to Proposition 4.1 in \cite{ChG} and provides the key estimates in this section. 

\begin{lemma}  \label{s6l63} There exist $0<r_1< \min(r_0, R/2)$  and an open neighborhood $\hat V \subset\subset \hat U$  of $\cS_{\tilde\epsilon}$ such that  the following estimates hold for any $r< r_1$.   

\begin{enumerate}

\item If $p\in \hat Z \setminus \hat V$, 
\begin{equation}\label{s634}
\max_{\{i\in\cI~|~ p  \in B_{i, 3R} \}}\nu_{i, \varphi_i} (p, r) \leq  2 \tilde \epsilon,
\end{equation}
and
\begin{equation}\label{est631} \max_{\{i\in\cI~|~ p  \in B_{i, 3R} \}}  \tilde\varphi_{i, r}(p)  > \sup_{\hat U}  \varphi_{\hat U} + 3 \tilde\epsilon \log r  +1  . 
\end{equation}

\item  If 
$$\max_{\{ i\in \cI \textnormal{ }|\textnormal{ } p \in B_{i, 3R} \}} \nu_{i, \varphi_i}(p, r) \geq 4 \tilde\epsilon, $$
then
\begin{equation}\label{est632}  \max_{\{i\in\cI~|~ p  \in B_{i, 3R} \}}  \tilde\varphi_{i, r} (p) \leq \inf_{\hat U}  \varphi_{\hat U} + 3 \tilde\epsilon \log r   - 1.
\end{equation}

\item If
$$\max_{\{ i\in \cI \textnormal{ }|\textnormal{ }x \in B_{i, 3R}\}} \nu_{i, \varphi_i}(p, r) \leq 4 \tilde\epsilon,$$
then
\begin{equation}\label{est633} 
\max_{\{j \in \cI~|~ p  \in B_{j, 3R}  \setminus B_{j, 2R} \}} \tilde\varphi_{j, r} (p)  <  \max_{\{i\in \cI~|~ x  \in B_{i, R}   \}}  \tilde\varphi_{i, r} (p)  - 2\tilde\epsilon.
\end{equation}

\end{enumerate}

\end{lemma}
 
\begin{proof}   We follow the argument in \cite{ChG} and prove the estimates respectively. 

\medskip

\noindent (1).  First we fix any open $\hat V \subset\subset \hat U$. Then we can pick $r_1$ sufficiently small so that (\ref{s634}) holds since the Lelong number of $\varphi$ at any point in $\hat Z \setminus \hat V$ is less than $\tilde\epsilon$ and $\nu_{\varphi}(\cdot, r)$ is decreasing in $r>0$ and upper semi-continuous for any fixed $r>0$.

 By assumption, there exists $i\in \cI$ such that  $\nu_{i, \varphi_i}(p, r) \leq  2\tilde\epsilon$ for $p\in \hat Z\setminus \hat V\cap B_{i, 3R}$, and  by Definition \ref{lelo} there exists $C_1=C_1(R, \varphi)>0$ such that 
$$  \varphi_{i, r} (p) \geq  2\tilde\epsilon \log r - C_1$$
for $0<r<r_1.$ By Lemma \ref{s6l62} (2), 
$$ \varphi_i^{(r)}(p) \geq  2\tilde\epsilon \log r - \eta \nu_{i, \varphi_i} (p, r) - C_1. $$
Therefore  
\begin{eqnarray*}
&&\tilde \varphi_{i, r}(p)    - \sup_{\hat U} \varphi_{\hat U} - 3 \tilde\epsilon \log r-1\\
&=&\varphi_i^{(r)}(p) -   (1-\delta) \phi_{i, \omega_0} (p) - \delta^3 \phi_{i, \theta} + \delta^2 \phi(p)- \sup_{\hat U} \varphi_{\hat U} - 3\tilde\epsilon \log r  -1\\
 &\geq& -\tilde\epsilon \log r - \eta \nu_{i, \varphi_i}(p, r) -  \sup_{B_{i, 4R}}  \phi_{i, \omega_0} - \delta^3 \sup_{B_{i, 4R}} \phi_{i, \theta}+ \delta^2 \inf_{\hat Z\setminus \hat V}\phi - \sup_{\hat U} \varphi_{\hat U}   -1 - C_1\\
&\geq& -\tilde\epsilon \log r - C_2
\end{eqnarray*}
for some fixed $C_2>0$ only depending on $R$, $\{B_{i, 4R} \}_{i\in \cI}$, $\phi_{i, \omega_0}$,  $\phi_{i, \theta}$, $K$,  $\varphi_{\hat U} $ and $\phi|_{\hat Z\setminus \hat V}$. 
Then (\ref{est631}) follows by choosing  $0<r< r_1< e^{- C_2/\tilde\epsilon}$.

\bigskip

\noindent (2).   By (2) in Lemma \ref{s6l62}, there exists $C_3=C_3(R, \varphi)>0$ such that  
$$\varphi_i^{(r)}(p)\leq \varphi_{i, r}(p) \leq 4 \tilde\epsilon \log r + C_3$$
for some $i\in \cI$. Therefore there exists $C_4=C_4(\{B_{i, 4R}\}_{i\in \cI}, \phi_{i, \omega_0}, \phi_{i, \theta}, \phi, \phi_{\hat U} )>0$
\begin{eqnarray*}
&&\tilde \varphi_{i, r}(p)    - \inf_{\hat U} \varphi_{\hat U} - 3 \tilde\epsilon \log r +1\\
&=&\varphi_i^{(r)}(p) -     (1-\delta) \phi_{i, \omega_0}(p)   -  \delta^3 \phi_{i, \theta}(p) + \delta^2 \phi (p) -  \varphi_{\hat U} (p)- 3 \tilde\epsilon \log r  +1\\
&\leq&  \tilde\epsilon \log r -  (1-\delta)  \inf_{B_{i, 4R}} \phi_{i, \omega_0}   -  \delta^3 \inf_{B_{i, 4R}} \phi_{i, \theta} + \delta^2 \sup_{\hat Z }\phi - \inf_{\hat U} \varphi_{\hat U} +1 + C_3\\
&\leq&  \tilde\epsilon \log r + C_4.
\end{eqnarray*}
  Then (\ref{est632}) follows by choosing  $0<r<r_1 <e^{- C_4/\tilde\epsilon}$.

 \bigskip

\noindent (3). Suppose 
$$p\in \left( B_{j, 3R} \setminus B_{j, 2R}  \right) \cap B_{i, R} $$
and
$$\max_{\{i'~|~ p\in B_{i',  3 R} \}} \nu_{i', \varphi}(p, r) \leq 4\tilde\epsilon $$
for $r< \min (r_0, R/2).$
Then at $p$,  
\begin{eqnarray*}
&&\tilde \varphi_{j, r}(p)  \\
&=& \varphi_{j, r}(p) -  (1-\delta) \phi_{j, \omega_0}(p) -  \delta^3 \phi_{j, \theta}(p) + \delta^2 \phi(p)\\
&\leq & \varphi_{j, \frac{r}{2}} (p) + \nu_{j, \varphi_j}(p, r) \log 2-  (1-\delta) \phi_{j,\omega_0} (p)-  \delta^3 \phi_{j, \theta}(p) + \delta^2 \phi(p)\\
&\leq & \sup_{B_{j, \frac{r}{2}(p)}} \varphi  +  (1-\delta) \left( \sup_{B_{j, \frac{r}{2}}(p)} \phi_{j, \omega_0}  - \phi_{j, \omega_0}(p) \right)    + 4\tilde\epsilon - \delta^3 \phi_{i, \theta} (p) -  \delta^3 \left(\phi_{j, \theta} (p) - \phi_{i, \theta}(p) \right) + \delta^2 \phi(p)\\
&\leq&\sup_{B_{i, r}(p) \subset B_{i, 2R} } \varphi+ 4\tilde\epsilon +  (1-\delta) Kr R -  \delta^3 \phi_{i, \theta} (p) -  \delta^3 ( 4R^2 - R^2) + \delta^2 \phi(p)\\
&\leq & \varphi_{i, r}(p) -  (1-\delta)\phi_{i, \omega_0} (p)  - \delta^3 \phi_{i, \theta}(p)+ \delta^2 \phi(p)+   (1-\delta) \sup_{B_{i, r}(p)}\left(\phi_{i, \omega_0}(p) -\phi_{i, \omega_0}  \right) \\
&&+ (1-\delta) K r R + 4\tilde\epsilon -  3\delta^3 R^2  \\
&\leq & \tilde\varphi_{i, r}(p) + (4+4 \eta) \tilde\epsilon+ 2KrR -  3\delta^3 R^2\\
&\leq & \tilde\varphi_{i, r}(p) - 2 \tilde \epsilon
\end{eqnarray*}
by the choice $\tilde \epsilon$ in (\ref{lelongcho}) and choosing $0<r< r_1<  \delta^3 R$.

\end{proof}

We define 
\begin{eqnarray*}
&&\tilde \varphi_\epsilon \\
&=&\tilde\max_{i\in \cI} \{ \tilde\varphi_{i, r},  \varphi_{\hat U}+ 3\tilde\epsilon \log r  \}\\
&=&\int\int_{\mathbb{R}^{\cI+1}} (\epsilon)^{\cI+1}\left(  \prod_{i=0}^\cI\rho\left(\frac{s_j}{\epsilon}\right) \right)\max_{ i\in \cI } \left( \tilde\varphi_{i, r} +s_i,   \varphi_{\hat U}+ 3\tilde\epsilon \log r +s_0 \right) ds_0ds_1...ds_{\cI}
\end{eqnarray*}
as the  regularized maximum of $\tilde\varphi_{i, r}$ and $(1-\delta)\varphi_{\hat U}$ for sufficiently small $\epsilon>0$ by assuming $\tilde\varphi_{i, r}\equiv0$ outside $B_{i, 3R}$ and $\varphi_{\hat U}+3\tilde\epsilon\log r \equiv 0$ outside $\hat U$. By Lemma \ref{s6l63}, $\tilde\varphi_\epsilon \in C^\infty(\hat Z)\cap \PSH(\hat Z, \omega_0)$  is well-defined for sufficiently small $r$ and $\epsilon$. Furhtermore, $\tilde \varphi$ coincides with $\varphi_{\hat U}$ in an open neighborhood of $\cS_{\hat Z}$ by the choice $\tilde\epsilon$.

Let  $\Omega_1 \in \alpha$  be the  closed $(1,1)$-current defined  by
$$ \Omega_1 = \omega_0 + \ddbar \tilde \varphi_\epsilon$$
for fixed sufficiently small $\epsilon>0$.
Then by construction, $\Omega_1 \in \alpha $ is smooth on $\hat Z$ 
$$\cP_\chi(\Omega_1)< n. $$
We can assume that $\cP_\chi(\Omega_1)< (1 - \epsilon')n $ on $\hat Z$ for some $\epsilon'>0$. We then let
$$\Omega_2 = (1-\epsilon')\Omega_1 + \epsilon' ( \omega_0 + \delta \ddbar \phi).$$
Then immediately we have the following lemma.
\begin{lemma} \label{s6l64} The K\"ahler current $\Omega_2 \in \alpha$ is smooth on $\hat Z \setminus  \cS_{\hat Z} $ and  
$$\cP_\chi(\Omega_2)  < n $$
or equivalently,
$$n \Omega_2^m - m \Omega_2^{m-1}\wedge \chi>0$$
on $\hat Z \setminus  \cS_{\hat Z}$.
Furhtermore, $\tilde\Omega$ has  positive Lelong number along $\cS_{\hat Z}$.

\end{lemma}
 
 \medskip

\begin{proof}[Proof of Theorem \ref{s6main}] If $Z$ is irreducible, Theorem \ref{s6main} follows immediately from Lemma \ref{s6l64} since $Z\setminus \cS_Z= \Phi\left( \hat Z \setminus \cS_{\hat Z}\right)$ and $\varphi_{\hat U}$ is the pullback of a $\omega_0$-PSH function on an open neighborhood of $U$. If $Z$ is not irreducible, we apply the construction of $\Phi: \cM' \rightarrow \cM$ by resolving singularities of $Z$ as in \S 4. Then the strict transform of $Z$ by $\Phi$ is a union of disjoint smooth $m$-dimensional submanifold of $\cX$. Then Theorem \ref{s6main} again follows by applying Lemma \ref{s6l64} to each component.

\end{proof}

\section{Proof of Theorem \ref{s3main}}

In this section, we will complete the proof of Theorem \ref{s3main} and its corollaries.

\begin{proof}[Proof of Theorem \ref{s3main}]

First we assume Lemma \ref{s4cont}. Let $Z$ be an $m$-dimensional analytic subvariety of $X$ and $\cS_Z$ be the singular set of $Z$ (including lower dimensional components of $Z$). By the induction assumption, there exists  an open neighborhood $U$ of $\cS_Z$ in $\cM$ such that there exists $\varphi_U\in C^\infty(U)\cap \PSH(U, \omega_0)$ satisfying
$$n(\omega_U)^{n-1} - (n-1)(\omega_U)^{n-2}\wedge \chi>0, ~\omega_U= \omega_0+\ddbar \varphi_U$$
 in $U$.

We choose $\varphi_Z$ as in Theorem \ref{s6main}. Then there exist $A>0$ and  open neighborhoods $U_0 \subset\subset U_1\subset\subset U_2 \subset \subset U$ of $\cS_Z $ in $\cM$ such that 

\begin{enumerate}

\item Both $ Z\setminus U_1$ and $ Z\setminus U_2$ are a union of finitely many $m$-dimensional smooth open analytic varieties of $Z$,

\medskip

\item $\varphi_Z <( \varphi_U-A)-2$ in $Z\cap U_1$, 

\medskip

\item $\varphi_Z > ( \varphi_U-A) +2$ in $Z \cap \left( U \setminus U_2\right)$, 

\end{enumerate}
since $\varphi_Z$ is smooth on $Z\setminus \cS_Z$ and tends to $-\infty$ uniformly along $\cS_Z$.

Let $\varphi_Z'$ be a smooth extension of $\varphi_Z$ in $Z\setminus U_1$ and let
$$\varphi''_Z  (p)= \varphi'_Z(p) + B d^2 (p)$$
for $B>1$, where $d$ is the distance function from $p$ to $Z$ with respect to any fixed K\"ahler metric on $\cM$. By choosing sufficiently large $B$, 
$\omega''_Z= \omega_0 + \ddbar \varphi''_Z$ is a K\"ahler form and 
$$n (\omega''_Z)^{n-1} - (n-1)(\omega''_Z)^{n-2} \wedge \chi >0$$
in an open neighborhood of $Z\setminus U_0$ in $X$. Since $\varphi''_Z$ is also a continuous extension of $\varphi'_Z$ and $\varphi_U$ is continuous, there exists a sufficiently small open neighborhood $U_3$ of $Z\setminus U_1$ in $\cM$ such that
\begin{enumerate}

\item $\varphi''_Z <( \varphi_U-A)-1$ in $U_3\cap U_1$ 

\item $\varphi''_Z > ( \varphi_U-A) +1$ in $U_3 \setminus U_2$ .

\end{enumerate}

We the define $\bar \varphi$ to be the regularized maximum of $\varphi''_Z$ and $\varphi_U$. Then
$$\omega= \omega_0+ \ddbar \bar \varphi$$
is a K\"ahler form in an open neighborhood of $X$ and it satisfies
$$n \omega^{n-1} - (n-1) \omega^{n-2}\wedge\chi >0$$
in an open neighborhood of $Z$. This proves Theorem \ref{s3main} by assuming Lemma \ref{s4cont}. 

To prove Lemma \ref{s4cont}, for any $\epsilon \in (0, \epsilon_Z)$, we let $$\cT_\epsilon=\{ t\in [0, T]~|~(\ref{332}) ~\textnormal{is solvable} \}, 
~\mathbf{t}_\epsilon= \inf \cT_\epsilon.$$ 
Obviously $\cT_\epsilon$ is open in $[0,T]$ and it suffices to show $\cT_\epsilon$ is closed or $\mathbf{t}_\epsilon=0$. Suppose $\mathbf{t}_\epsilon>0$, then we can apply the same argument in \S 4, \S 5 and \S 6 with suitable small changes and let $t\rightarrow \mathbf{t}_\epsilon$ decreasingly to obtain a K\"ahler form $\omega\in \hat\alpha= \hat\alpha(\mathbf{t}_\epsilon, \epsilon)$ satisfying 
$$n(\omega')^n - (n-1)(\omega')^{n-2}\wedge \chi>0$$ in an open neighborhood of any given $m$-dimensional subvariety $Z$ of $X$. But such a subsolution $\omega'$ implies that equation (\ref{332}) can be solved at $\mathbf{t}_\epsilon$ and this leads to contradiction.

\end{proof}

Theorem \ref{s3main} is a special case of Theorem \ref{main1} by assuming $X$ is irreducible. However,  Theorem \ref{main1} can be proved by the same argument for Theorem \ref{s3main}.
Corollary \ref{main2} is an immediate consequence of Theorem \ref{main1} and Corollary \ref{main3} follows by combining Corollary \ref{main2} and the result of \cite{SW1}.

\section{Applications}

In this section, we will prove Theorem \ref{maincsc} by combining the ideas from \cite{Ch1, W2, SW1, LSY, JSS, CC}.  

Let $X$ be a K\"ahler manifold of $\dim X=n$. Tian's $\alpha$-invariant  for a K\"ahler class $\gamma$ on $X$ is defined by 
\begin{equation}\label{alphainv}
\alpha(X, \gamma) = \sup\{ \kappa >0~|~\int_X e^{-\kappa(\varphi - \sup_X \varphi)} \omega^n \leq C_\kappa, ~\forall \varphi\in \PSH(X, \omega)\},
\end{equation}
for any given K\"ahler form $\omega\in \gamma$. We remark that the $\alpha$-invariant does not depend on the choice of $\omega \in \alpha$ and it is always positive. 

The following lemma is proved in \cite{JSS} (Lemma 2.1) after applying the important work of Chen-Cheng \cite{CC} relating properness of the Mabuchi $K$-energy and existence of constant scalar curvature K\"ahler metrics. 

\begin{lemma} \label{s8jss} Let $X$ be an $n$-dimensional compact K\"ahler manifold and let $\gamma$ be a K\"ahler class on $X$ satisfying
$$\gamma^n = \gamma^{n-1}\cdot K_X.$$ If there exist $\epsilon \in \left(0, \frac{n+1}{n} \alpha(X, \gamma)\right)$, a K\"ahler from $\omega' \in [\omega]$ and a closed $(1,1)$-form $\eta \in [K_X]$ such that 
$$\eta + \epsilon \omega'>0$$
and
$$\left( (n+\epsilon) \omega' - (n-1)\eta \right) \wedge (\omega')^{n-2} >0$$
everywhere on $X$, then the Mabuchi $K$-energy is proper on $\PSH(X, \omega')\cap C^\infty(X)$. In particular, there exists a unique cscK  metric in $\gamma$.

\end{lemma}

Now we can apply Theorem \ref{main1} and Lemma \ref{s8jss} to prove  the following theorem (Theorem \ref{maincsc}).

\begin{proof}[Proof of Theorem \ref{maincsc}.] If  $K_X \cdot \gamma^{n-1} >0$, we can assume that $K_X \cdot \gamma^{n-1} = \gamma^n$ by rescaling $\gamma$. Then 
$$\frac{ (K_X + \epsilon \gamma)\cdot \gamma^{n-1} }{ \gamma^n} = 1+\epsilon.$$
We let 
$$\gamma_\epsilon = (1+\epsilon) \gamma, ~~\beta_\epsilon = K_X + \epsilon \gamma.$$
 Then both $\gamma_\epsilon$ and $\beta_\epsilon$ are K\"ahler for any $\epsilon\in (0, \alpha(X, \gamma))$ and $\gamma_\epsilon^n = \gamma_\epsilon^{n-1} \cdot \beta_\epsilon$.
Furthermore, for any $m$-dimensional analytic subvariety $Z$ of $X$
\begin{eqnarray*}
&& \left( n \gamma_\epsilon^m - m \gamma_\epsilon^{m-1} \cdot \beta_\epsilon \right) \cdot Z\\
&=&  (1+\epsilon)^{m-1} \left(  (n+ (n-m) \epsilon ) \gamma  - m K_X  \right) \cdot \gamma^{m-1}\cdot Z\\
&>&0
\end{eqnarray*}
by the assumption of Theorem \ref{maincsc}. We can now apply Theorem \ref{main1} or Corollary \ref{main2} and  so there exist K\"ahler forms $\omega_\epsilon\in \gamma_\epsilon$ and $\chi_\epsilon \in \beta_\epsilon$ such that
\begin{equation}\label{cscpos1}
 \left( n \omega_\epsilon - (n-1) \chi_\epsilon \right)  \wedge \omega_\epsilon^{n-2} >0.
 \end{equation}
We let $$\omega'_\epsilon =(1+\epsilon)^{-1} \omega_\epsilon\in \gamma, ~~\eta_\epsilon  = \chi_\epsilon -  \epsilon \omega'_\epsilon \in [K_X]. $$
Then $\eta_\epsilon$ is K\"ahler and (\ref{cscpos1}) is equivalent to
$$\left(  (n+\epsilon) \omega'_\epsilon - (n-1)\eta_\epsilon \right) \wedge (\omega'_\epsilon)^{n-2} >0.$$
Then we can apply Lemma \ref{s8jss} and so there exists a unique cscK  metric in $\gamma$.

If $K_X\cdot \gamma^{n-1} =0$, then we claim that $K_X\cdot \gamma^{m-1}\cdot Z=0$ for any $m$-dimensional analytic subvariety $Z$ of $X$ by the following argument. By Poincare duality, there exists a smooth closed real $(n-m, n-m)$-form $\Omega$ such that
$$K_X \cdot \gamma^{m-1}\cdot Z = \int_X \eta\wedge \theta^{m-1} \wedge \Omega\geq 0$$
since $K_X$ is nef, where $\eta$ is smooth closed $(1,1)$-form in $[K_X]$ and $\theta$ is a K\"ahler form in $\gamma$. Without loss of generality, we can assume $\Omega \leq C \theta^m$ for some fixed $C>0$.  For any $\epsilon>0$, we choose $\eta_\epsilon$ to be a K\"ahler from in $K_X+ \epsilon \gamma$. Then
$$K_X \cdot \gamma^{m-1}\cdot Z =\lim_{\epsilon\rightarrow 0} \int_X\eta_\epsilon \wedge \theta^n \wedge \Omega \leq C \lim_{\epsilon\rightarrow 0} \int_X \eta_\epsilon\wedge \theta^{n-1}=CK_X\cdot \gamma^{n-1}=0.$$

Now we let  
$ \beta = K_X + \gamma. $
Then $\gamma^n = \gamma ^{n-1} \cdot \beta $ and for any $m$-dimensional analytic subvariety $Z$ of $X$
\begin{eqnarray*}
&& \left( n \gamma^m - m \gamma^{m-1} \cdot \beta \right) \cdot Z\\
&=&   \left(  (n-m  ) \gamma  - m K_X  \right) \cdot \gamma^{m-1}\cdot Z\\
&=&  (n-m)  \gamma^m\cdot Z\\
&>&0.
\end{eqnarray*}
As we argued before, there exist K\"ahler forms $\omega  \in \gamma $ and $\chi  \in \beta $ satisfying
$$n \omega^{n-1}- (n-1) \omega^{n-2} \wedge \chi.$$ 
By letting $\eta=  \chi  - \omega $, we have
$$(\omega   - (n-1) \eta ) \wedge  \omega ^{n-2}>0.$$ 
We can again apply Lemma \ref{s8jss} and complete the proof of Theorem \ref{maincsc}.
\end{proof}

\begin{proof}[Proof of Corollary \ref{maincsc2}.]  Let $\gamma$ be a K\"ahler class on $X$ and $\gamma_\epsilon = K_X + \epsilon \gamma$ for $\epsilon>0$.  

We first assume that $\gamma^{n-1}\cdot K_X>0$. Then for any $m$-dimensional analytic subvariety $Z$ with $m<n$,  
\begin{eqnarray*}
&& m\left. \frac{ \gamma_\epsilon^{m-1}\cdot K_X }{  \gamma_\epsilon^m  } \right|_Z \\
&=& \left. m \frac{ \sum_{l=1}^{m}  \begin{pmatrix} m-1\\ l-1 \end{pmatrix}\epsilon^{m-l} \gamma^{m-l}  \cdot (K_X)^l }{\sum_{l=0}^m  \begin{pmatrix} m\\ l \end{pmatrix}\epsilon^{m-l} \gamma^{m-l}  \cdot (K_X)^l  }  \right|_Z   \\
&\leq& m\max_{l=1, ..., m} \frac{  \begin{pmatrix} m-1\\ l-1 \end{pmatrix} }{\begin{pmatrix} m\\ l \end{pmatrix}}\\
&=&\max_{l=1, ..., m} l \\
&=&m.
\end{eqnarray*}
On the other hand, if we let $\nu = \max\{l\geq 0~|~\gamma^{n-l} \cdot (K_X)^l >0\}$ be the numerical dimension of $K_X$, then 
\begin{eqnarray*}
&& n  \frac{ \gamma_\epsilon^{n-1}\cdot K_X }{  \gamma_\epsilon^n  }   \\
&=& n \frac{ \sum_{l=1}^{\nu}  \begin{pmatrix} n-1\\ l-1 \end{pmatrix}\epsilon^{n-l} \gamma^{n-l}  \cdot (K_X)^l }{\sum_{l=0}^{\nu}  \begin{pmatrix} n\\ l \end{pmatrix}\epsilon^{n-l} \gamma^{n-l}  \cdot (K_X)^l  }      \\
&=& n + O(\epsilon).
\end{eqnarray*}
Then for sufficiently small $\epsilon$, $\gamma_\epsilon$ satisfies the estimate (\ref{slopecsc}) in Theorem \ref{maincsc} for any analytic subvariety $Z$ of $X$.

If $\gamma^{n-1}\cdot K_X=0$, then $K_X\cdot \gamma^{m-1}\cdot Z=0$ for any $m$-dimensional analytic subvariety $Z$ of $X$. The corollary immediately follows from Theorem \ref{maincsc}.

\end{proof}
 
\bigskip
\bigskip

\noindent {\bf{Acknowledgements:}} The author would like to thank Ved Datar, Slawomir Dinew and Xin Fu for stimulating discussions.

\end{document}